\documentclass[11pt]{article}
\usepackage{fullpage,amsmath,amssymb,color,hyperref}
\usepackage[T1]{fontenc}
\usepackage[utf8]{inputenc}

\usepackage{amsmath,amssymb,amsfonts,amsthm,amstext,amsopn}
\usepackage{latexsym} 

\usepackage{algorithm}
\usepackage{algorithmic} 

\usepackage{graphicx}
\usepackage{overpic}
\usepackage{epstopdf} 

\DeclareGraphicsExtensions{.pdf}
\usepackage[table]{xcolor}
\usepackage{tcolorbox}
\usepackage{colortbl} 

\usepackage{array}
\usepackage{booktabs,makecell}
\usepackage{multirow}
\usepackage{diagbox}
\usepackage{easybmat}
\usepackage{blkarray}

\usepackage{hyperref}
\usepackage{xr-hyper}   
\usepackage[toc,page]{appendix}

\usepackage{tikz}
\usetikzlibrary{arrows,matrix,positioning,calc,decorations.pathmorphing}

\usepackage{lipsum}
\usepackage{verbatim}
\usepackage[all,cmtip]{xy}
\usepackage{chemfig}
\usepackage{siunitx}
\usepackage{titlesec}
\titleformat*{\section}{\large\bfseries}
\titleformat*{\subsection}{\bfseries}
\usepackage{makeidx,wrapfig}
\makeindex
\usepackage{pstricks,framed} 
\usepackage{mathdots}
\usepackage{url} 
\usepackage{titlesec}

\titlespacing*{\section}{0pt}{0.1ex}{0.1ex}
\titlespacing*{\subsection}{0pt}{.1ex}{0.1ex}

\newtheorem{theorem}{Theorem}
\newtheorem{lemma}{Lemma}

\newtheorem{proposition}{Proposition}
\newtheorem{assumption}{Assumption}
\newtheorem{corollary}{Corollary}

\newcommand{\EE}{\mathbb{E}}

\newcommand{\R}{\mathbb{R}}

\newcommand{\vdot}[2]{\left\langle #1, #2 \right\rangle}

\DeclareMathOperator*{\minimize}{minimize}

\DeclareMathOperator*{\argmin}{arg\,min}

\DeclareMathOperator{\tr}{tr}

\newcommand{\bbe}{\mathbb{E}}

\newcommand{\bbr}{\mathbb{R}}
\def\w{\omega}

\def\eqnok#1{(\ref{#1})}
\def\vgap{\vspace*{.1in}}

\newcommand{\beq}{\begin{equation}}
\newcommand{\eeq}{\end{equation}}
\newcommand{\beqa}{\begin{eqnarray}}
\newcommand{\eeqa}{\end{eqnarray}}
\newcommand{\beqas}{\begin{eqnarray*}}
\newcommand{\eeqas}{\end{eqnarray*}}
\newcommand{\bi}{\begin{itemize}}
\newcommand{\ei}{\end{itemize}}
\newcommand{\ba}{\begin{array}}
\newcommand{\ea}{\end{array}}

\newcommand{\nn}{\nonumber}
\NewEnviron{myalign*}{%
  \begingroup
  \setlength{\abovedisplayskip}{2pt}%
  \setlength{\belowdisplayskip}{2pt}%
  \setlength{\abovedisplayshortskip}{2pt}%
  \setlength{\belowdisplayshortskip}{2pt}%
  \begin{align*}
    \BODY
  \end{align*}
  \endgroup
}

\NewEnviron{myalign}{%
  \begingroup
  \setlength{\abovedisplayskip}{2pt}%
  \setlength{\belowdisplayskip}{2pt}%
  \setlength{\abovedisplayshortskip}{2pt}%
  \setlength{\belowdisplayshortskip}{2pt}%
  \begin{align}
    \BODY
  \end{align}
  \endgroup
}

\NewEnviron{myequation*}{%
  \begingroup
  \setlength{\abovedisplayskip}{2pt}%
  \setlength{\belowdisplayskip}{2pt}%
  \setlength{\abovedisplayshortskip}{2pt}%
  \setlength{\belowdisplayshortskip}{2pt}%
  \begin{equation*}
    \BODY
  \end{equation*}
  \endgroup
}

\NewEnviron{myequation}{%
  \begingroup
  \setlength{\abovedisplayskip}{2pt}%
  \setlength{\belowdisplayskip}{2pt}%
  \setlength{\abovedisplayshortskip}{2pt}%
  \setlength{\belowdisplayshortskip}{2pt}%
  \begin{equation}
    \BODY
  \end{equation}
  \endgroup
}

\usepackage[font=small,labelfont=bf,tableposition=top]{caption}

\tikzset{
  snake it/.style={-stealth,
    decoration={snake, amplitude=1.5mm, segment length=2.5mm, post length=2.9mm},
    decorate}
}

\title{Optimal Zeroth-Order Bilevel Optimization\thanks{The authors have contributed equally to this work.
}}

\author{Alireza Aghasi\thanks{Department of Electrical Engineering \& Computer Science, Oregon State University, Email: 
  \texttt{alireza.aghasi@oregonstate.edu}.}
\and Jeongyeol Kwon\thanks{Wisconsin Institute for Discovery, University of Wisconsin-Madison, Email: \texttt{sortingkwon@gmail.com}.}
\and Saeed Ghadimi\thanks{Department of Management Science \& Engineering, University of Waterloo, Email: \texttt{sghadimi@uwaterloo.ca}.}
}

\usepackage{amsopn}

\begin{document}
\maketitle

\begin{abstract}
		In this paper, we develop zeroth-order algorithms with provably (nearly) optimal sample complexity for stochastic bilevel optimization, where only noisy function evaluations are available. We propose two distinct algorithms: the first is inspired by Jacobian/Hessian-based approaches, and the second builds on using a penalty function reformulation. The Jacobian/Hessian-based method achieves a sample complexity of $\mathcal{O}(d^3/\epsilon^2)$, which is optimal in terms of accuracy $\epsilon$, albeit with polynomial dependence on the problem dimension $d$. In contrast, the penalty-based method sharpens this guarantee to $\mathcal{O}(d/\epsilon^2)$, optimally reducing the dimension dependence to linear while preserving optimal accuracy scaling. Our analysis is built upon Gaussian smoothing techniques, and we rigorously establish their validity under the stochastic bilevel settings considered in the existing literature. To the best of our knowledge, this is the first work to provide provably optimal sample complexity guarantees for a zeroth-order stochastic approximation method in bilevel optimization.
\end{abstract}


\section{Introduction}\label{sec:intro}
We study the stochastic bilevel programming (BLP) problem of the form
	\begin{align} \label{main_prob_st}
		&\minimize_{x \; \in \; X\subseteq \mathbb{R}^n} \left\{\psi(x):= f(x,y^*(x))=\bbe_\xi [F(x,y^*(x),\xi)]\right\} \nn \\
		& \text{subject to:} \ \  y^*(x) = \argmin_{y \in \bbr^m} \left\{g(x,y)= \bbe_\zeta [G(x,y,\zeta)]\right\},
	\end{align}
	where \( f \) and \( g \) are continuously differentiable over \((x,y) \in X \times \mathbb{R}^m\), \(X \) is a closed convex set, and \(\xi\) and \(\zeta\) are independent random vectors, whose distributions reside in subspaces of possibly different dimensions.
	We focus on a challenging noisy, derivative-free setting in which unbiased stochastic gradients of \(f\) and \(g\) are unavailable, and the only accessible information comes from noisy oracle evaluations of \(f\) and \(g\).
	
	Bilevel programming (BLP) is a fundamental problem in engineering and economics arising in many applications including decision making \cite{lu2016multilevel}, game theory \cite{yue2017stackelberg}, and optimal design \cite{herskovits2000contact}. Recently, BLP has been widely adopted in many machine learning and artificial intelligence problems that involve hierarchical two-level structures. A prominent example is meta-learning \cite{franceschi2018bilevel,hospedales2021meta}, in which knowledge from prior tasks accelerates and automates learning on new tasks. Effective meta-learning models generalize to unseen tasks and require significantly fewer training examples (an ability often referred to as learning to learn \cite{hospedales2021meta}). Beyond meta-learning, BLP is central to model selection, where it identifies parameters by minimizing generalization error estimates \cite{giovannelli2021bilevel}, and to hyperparameter optimization, which enables simultaneous model training and optimal parameter tuning \cite{franceschi2018bilevel}. In automated machine learning (AutoML) \cite{hutter2019automated}, BLP often follows a leader–follower format where the follower takes the optimal action that maximally benefits the leader’s objective value. In generative adversarial networks (GANs), originally formulated as a zero-sum, min-max game, a bilevel formulation can extend the problem to a non-saturating, non-zero-sum setting, yielding more stable numerical behavior and stronger gradients during early training \cite{gidelvariational}.
	
	Other applications include continual learning \cite{borsos2020coresets}, where shared parameters facilitate adaptation without forgetting previous data patterns; actor–critic methods in reinforcement learning \cite{konda1999actor}; multi-agent coordination \cite{zhang2020bi}; and neural architecture search, where differentiable BLP approaches achieve high performance and efficiency \cite{liu2018darts}. Developing scalable solution strategies for BLP can thus advance major areas such as meta-learning and AutoML, while also providing foundational tools for more complex settings involving multiple lower-level subproblems.
	
	Solving bilevel programs is computationally challenging, particularly in large-scale machine learning applications. Gradient-based methods require computing or estimating the hypergradient of $\psi$ at a given point $x$, which under some regularity conditions is given by \cite{GhadWang18}:
	\begin{equation}\label{grad_f}
		\nabla \psi(x) = \nabla_x f\big(x, y^*(x)\big) 
		- \nabla_{xy}^2 g\big(x, y^*(x)\big) 
		\left[ \nabla_{yy}^2 g\big(x, y^*(x)\big) \right]^{-1} 
		\nabla_y f\big(x, y^*(x)\big).
	\end{equation}
	Two general difficulties arise in computing \eqref{grad_f}. First, for each $x$, one must solve the lower-level problem to obtain $y^*(x)$, often requiring multiple updates of $y$ before a single update of $x$. Second, the formula involves second-order derivatives of $g$, necessitating access to, or accurate estimation of, Jacobians and inverse Hessians. Existing methods typically rely on explicit second-order information and focus on efficient stochastic estimation under noise \cite{ji2021bilevel,chen2022single}.
	
	In this work, we focus on scalable zeroth-order methods for solving \eqref{main_prob_st}, relying solely on noisy oracle evaluations of \(f\) and \(g\). Such methods are well suited to optimization problems, including BLP, where gradients are unavailable or impractical to compute due to model complexity, lack of analytical access, or computational constraints. In these cases, gradients are typically approximated via deterministic or randomized finite-difference schemes \cite{shi2021numerical,nesterov2017random}. 
	
	Zeroth-order approaches have found broad application across diverse engineering domains \cite{aghasi2013geometric}. In the realm of machine learning, zeroth-order  methods become essential when gradients are inaccessible due to privacy, model complexity, or computational constraints \cite{liu2020primer}. These methods have been considered in adversarial machine learning for black-box attacks \cite{ilyas2018black}, where adversaries craft adversarial examples via query–response interactions without internal model access, as well as for defending and robustifying such models \cite{zhang2021robustify}. These methods also enhance interpretability by explaining black-box predictions using only output probabilities \cite{dhurandhar2019model}, and are employed in meta-learning to reduce query costs \cite{du2019query}, learn optimizers \cite{ruan2019learning}, and adapt meta-learners through finite-difference or Gaussian smoothing techniques \cite{nesterov2017random}. 
	
	\paragraph{Previous BLP Methods and the Contributions of this Work} Several general techniques address the computational challenges of bilevel programming (BLP). A common approach reduces the BLP to a single-level problem by replacing the lower-level problem with its optimality conditions (e.g., \cite{Hansen1992NewBR}). For large-scale lower problems, however, this often results in an excessive number of constraints and additional computational burdens. Alternatively, iterative algorithms directly solve the BLP, including penalty methods (\cite{Aiyoshi1980HIERARCHICALDS,Case1998AnLP}), descent methods (\cite{Falk1995OnBP,Kolstad1990DerivativeEA}), and trust-region methods with mixed-integer subproblems (\cite{Colson2005ATM}). Most existing works on bilevel optimization develop (stochastic) methods requiring second-order derivatives of the inner function $g$ (e.g., \cite{chen2023optimal,GhadWang18,ji2021provably}), though more recent efforts have introduced fully (stochastic) first-order methods that only require first derivatives of both upper- and lower-level objectives (\cite{Ye2022BOMEBO,kwon2023fully}).
	
	The computational challenges of solving BLPs are further compounded when only noisy function queries are available and derivatives of $f$ and $g$ are inaccessible. Literature on derivative-free bilevel optimization is scarce and, to our knowledge, suboptimal in terms of scalability and accuracy. For example, \cite{conn2012bilevel,ehrhardt2021inexact} address a quadratic belief model, but their methods apply mainly to deterministic (non-stochastic) settings and are not scalable. A classical zeroth-order approach approximates functions via convolution with a probability density (e.g., uniform over a ball or Gaussian), whose gradient can be computed from function values, enabling biased gradient estimation for $f$ and $g$ in \eqnok{main_prob_st} without explicit derivatives. In \cite{Gu-etal2021}, Gaussian convolution is applied to the upper-level function, assuming the lower-level problem is solved efficiently via methods such as (proximal) gradient descent; however, no convergence analysis is provided. More recently, Jacobian/Hessian-free methods have been proposed that avoid second-order information on $g$ but still require first-order derivatives of both $f$ and $g$ (e.g., \cite{sow2022convergence,yang2023achieving}).
	
	Finally, \cite{AghaGhad25} introduces the first fully zeroth-order stochastic approximation framework for BLPs, requiring no unbiased first- or second-order derivatives for either level. The authors provide non-asymptotic convergence analysis of the proposed algorithm under different convexity assumptions on $\psi(x)$, though these remain suboptimal with respect to dimension and accuracy. In particular, they show that when $\psi(x)$ is non-convex, their proposed method possesses a sample complexity of ${\cal O}\left(\frac{(n+m)^4}{\epsilon^3 }\log \left(\frac{n+m}{\epsilon}\right)\right)$ to find an $\epsilon$-stationary point of the BLP in \eqref{main_prob_st} (a point $\bar x \in \mathbb{R}^n$ such that $\bbe[\|\nabla \psi(\bar x)\|^2] \le \epsilon$ assuming $X \equiv \mathbb{R}^n$ for simplicity).
	
	Our contribution in this paper is focused on improving the available complexity bounds for zeroth-order BLP problems. First, we modify the algorithm presented in \cite{AghaGhad25} by taking mini-batch of samples to estimate derivatives of the objective functions using their noisy evaluations and different choice of parameters. We then show that our modified algorithm can find an $\epsilon$-stationary point of the BLP with at most ${\cal O} \left(\frac{m(m+n)^2 \log(1/\epsilon)}{\epsilon^2 } \right)$ number of evaluations of $f$ and $g$ which is significantly better than the one in \cite{AghaGhad25} in terms of both problem dimension and target accuracy. This complexity bound is indeed optimal in terms of dependence on $\epsilon$ (up to a logarithmic factor) while still cubically depends on the problem dimension. This dependency seems unimprovable when we are estimating second-order derivatives with noisy function evaluations. As our second contribution, motivated by \cite{kwon2024complexity}, we provide another fully zeroth-order algorithm for solving BLP in which we don't need to estimate Hessian or its inverse, and we only estimate first-order derivatives of both objective function. We establish the finite-time convergence analysis of this algorithm and show that the aforementioned sample complexity can be improved to ${\cal O} \left(\frac{m+n}{\epsilon^2 } \right)$. To the best of our knowledge this is the first time that such an optimal sample complexity in terms of dependence on both problem dimension and target accuracy is provided in the literature for fully zeroth-order stochastic optimization
	algorithms for solving BLP. 
	
	The remainder of the paper is organized as follows. We conclude this section by providing our main assumptions and review of zeroth-order estimates. We then provide our nearly optimal algorithm for solving the BLP in Section~\ref{near_opt} followed by our optimal method in Section~\ref{full_opt}. Finally, we conduct some numerical experiments in Section~\ref{numeric} to show practical performance of the proposed methods.

	\subsection{Main Assumptions}
	Throughout the paper, we make different assumptions which are common in the literature of BLP.
	
	\begin{assumption}\label{fg_assumption}
		The following statements hold for the functions $f$ and $g$.
		\begin{itemize}
			\item [a)] The function $f$ and its gradient are Lipschitz continuous. 
			
			\item [b)] The function $g$ is twice differentiable in $(x,y)$ with Lipschitz continuous gradient and Hessian. Moreover, for any fixed $ x \in X$, $g(x,\cdot)$ is strongly convex.
		\end{itemize}
	\end{assumption}
	Since we are dealing with zeroth-order estimates, we also need to make the following assumptions about the stochastic estimators of $f$ and $g$.
	\begin{assumption}\label{stochastic_assumption}
		The stochastic functions $F(x,y,\xi)$ and $G(x,y,\zeta)$ are respectively differentiable and twice differentiable in $(x,y) \in \bbr^{n \times m}$, and Assumption~\ref{fg_assumption} statements hold for them. Moreover, $\nabla_x F(x,y,\xi)$, $\nabla_y F(x,y,\xi)$, $\nabla_y G(x,y,\zeta)$, $\nabla^2_{xy} G(x,y,\zeta)$, and $\nabla^2_{yy} G(x,y,\zeta)$ are unbiased estimators with bounded variance for the true gradients and Hessian of $f$ and $g$.
	\end{assumption}
	We also need the following boundedness assumption to establish the convergence analysis of our proposed method.
	\begin{assumption}\label{bnd_yx}
		The optimal solution to the lower problem is bounded, in other words, $\max_{x \in X} \|y^*(x)\|$ is bounded.
	\end{assumption}
	
	When only dealing with estimating first-order derivatives, we also need the following assumption while relaxing the the above boundedness assumption.
	\begin{assumption}
		\label{assumption:fourth_moment}
		The $4^{th}$-order moment of stochastic partial derivatives of $f$ w.r.t $y$, $\nabla_y F (x,y; \zeta)$, is bounded.  
	\end{assumption}

\subsection{Review of Zeroth-order Approximations}

	In this subsection, we review some tools and concepts from zeroth-order approximation. We first start with the Stein's identity theorem \cite{stein1972bound,stein1981estimation} which is foundation for our zeroth-order approximations.
	\begin{theorem}\label{thStein}
		Let $u\sim\mathcal{N}(0,I_d)$, be a standard Gaussian random vector, and let $q:\mathbb{R}^d\to\mathbb{R}$, be an almost-differentiable function with $\mathbb{E}[\|\nabla q\|]<\infty$, then $\mathbb{E}[u~\!q(u)] = \mathbb{E}[\nabla q(u)]$.
		Furthermore, when the function $q$ has a twice continuously differentiable Hessian, we have $\mathbb{E}\left[\left( u u^\top - I_d\right)q(u) \right]  = \mathbb{E}\left[\nabla^2 q(u)\right].$
	\end{theorem}
	Following the above theorem, and the block structure $z = (x,y)$ in BLP, one can consider a smooth approximation to $q(x,y)$ as
	\begin{equation}
		q_{\eta,\mu}(x,y) \triangleq \mathbb{E}_{u,v} [q(x+\eta u, y+\mu v)],
	\end{equation}
	where $u\sim\mathcal{N}(0,I_n)$, $v\sim\mathcal{N}(0,I_m)$, and $\eta\geq 0$ and $\mu\geq 0$ denote the level of smoothness along the $x$ and $y$ directions, respectively (see \cite{nesterov2017random,AghaGhad25} for the properties of the smoothed function and its gradient with one and two blocks of variables, respectively).
	
	Denoting, $q_{\eta,\mu}(x,y) \triangleq \mathbb{E}_{u,v,\zeta} [Q(x+\eta u, y+\mu v,\zeta)$, where $\zeta$ is a random vector independent of the Gaussian random vectors $(u,v)$, we can introduce zeroth-order approximation of derivatives of $q$ as follows.
	\begin{equation}\label{stoch:gradx_minibatch}
	    \tilde \nabla_x Q_{\eta,\mu}(x,y,\{\zeta_i\}_{i=1}^N)  =  \frac{1}{N}\sum_{i=1}^N \frac{Q(x+\eta u_i, y+\mu v_i,\zeta_i) - Q(x, y,\zeta_i)}{\eta} \ u_i,
	\end{equation}
    \begin{equation}\label{stoch:grady_minibatch}
        \tilde \nabla_y Q_{\eta,\mu}(x,y,\{\zeta_i\}_{i=1}^N)  = \frac{1}{N} \sum_{i=1}^N \frac{Q(x+\eta u_i, y+\mu v_i, \zeta_i) - Q(x, y,\zeta_i)}{\mu} \ v_i,
    \end{equation}
	\begin{align}\notag 
		\tilde \nabla_{xy}^2 &Q_{\eta,\mu}(x,y,\{\zeta_i\}_{i=1}^N)=\\ \label{stochHessxy_minibatch}  & \frac{1}{N}\sum_{i=1}^N u_i v_i^\top  \frac{Q(x+\eta u_i, y+\mu v_i,\zeta_i) + Q(x-\eta u_i,y-\mu v_i,\zeta_i)-2Q(x,y,\zeta_i )}{2\eta\mu},
	\end{align}
    \begin{align}\notag 
		\tilde \nabla_{xx}^2 &Q_{\eta,\mu}(x,y,\{\zeta_i\}_{i=1}^N)=\\ \label{stochHessxx_minibatch} &\! \frac{1}{N}\!\sum_{i=1}^N (u_i u_i^\top \!\! -\! I)\frac{Q(x+\eta u_i, y+\mu v_i,\zeta_i) + Q(x-\eta u_i,y-\mu v_i,\zeta_i)-2Q(x,y,\zeta_i )}{2\eta^2}.
	\end{align}
	In the above estimates, we use a batch of size $N$ to reduce the error associated with estimating the derivatives. Below, we provide upper bounds on the above first- and second-order derivative estimates.
	\begin{proposition}\label{propNestApprox_stch}
		Consider $x\in\mathbb{R}^n$ and $y\in \mathbb{R}^m$. For a given function $q: \mathbb{R}^{n+m}\to \mathbb{R}$, and positive scalars $\eta$ and $\mu$.
		\begin{itemize}
			\item[(a)]If $q\in \mathcal{C}^1(\mathbb{R}^{n+m};L_{1,q})$, we have
			\begin{align*}
				\EE &\|\tilde \nabla_x Q_{\eta,\mu}(x,y,\{\zeta_i\}_{i=1}^N) - \nabla_x q_{\eta,\mu}(x,y)\|^2 \le \EE \|\tilde \nabla_x Q_{\eta,\mu}(x,y,\{\zeta_i\}_{i=1}^N) \|^2\\ &\le \frac{L_{1,Q}^2}{N}\left(\eta^2(n+6)^3+\frac{\mu^4}{\eta^2}n(m+4)^2 \right) +  \frac{4(n+2)}{N}\left(\sigma_{1,Q}^2+ \|\nabla_x  q(x, y)\|^2\right) \\& ~~~+  \frac{4\mu^2}{\eta^2N}n\left(\sigma_{1,Q}^2+\|\nabla_y  q(x, y)\|^2\right),
			\end{align*}
			and
			\begin{align*}
				\EE &\|\tilde \nabla_y Q_{\eta,\mu}(x,y,\{\zeta_i\}_{i=1}^N) - \nabla_y q_{\eta,\mu}(x,y)\|^2 \le \EE \|\tilde \nabla_y Q_{\eta,\mu}(x,y,\{\zeta_i\}_{i=1}^N) \|^2\\ &\le \frac{L_{1,Q}^2}{N}\left( \mu^2(m+6)^3+\frac{\eta^4}{\mu^2}m(n+4)^2  \right) +  \frac{4\eta^2}{\mu^2 N}m\left(\sigma_{1,Q}^2+ \|\nabla_x  q(x, y)\|^2\right) \\&~~~+  \frac{4(m+2)}{N}\left(\sigma_{1,Q}^2+\|\nabla_y  q(x, y)\|^2\right),
			\end{align*}
			where $\tilde \nabla_x Q_{\eta,\mu}(x,y,\{\zeta_i\}_{i=1}^N)$ and $\tilde \nabla_y Q_{\eta,\mu}(x,y,\{\zeta_i\}_{i=1}^N)$ are defined in \eqref{stoch:gradx_minibatch} and \eqref{stoch:grady_minibatch}, respectively.
			
			\item [(b)] If $q\in \mathcal{C}^2(\mathbb{R}^{n+m};L_{2,q})$, then
			\begin{align*}
				\EE&_{\{u_i,v_i,\zeta_i\}_{i=1}^N\mid \theta }  \|\tilde \nabla^2_{xy} Q_{\eta,\mu}(x,y,\{\zeta_i\}_{i=1}^N)\theta - \nabla_{xy} q_{\eta,\mu}(x,y)\theta\|^2  \\&\leq  ~\frac{8L_{2,Q}^2}{N}\left[ \frac{\eta^4}{\mu^2}(n+8)^4 + \frac{2\mu^4}{\eta^2} n(m+12)^3\right]\|\theta\|^2\\ \notag &~~~+\frac{1}{N} \bigg[ \frac{6\eta^2}{\mu^2}(n+4)(n+2)\left(\sigma_{2,Q}^2 + \|\nabla^2_{xx} q(x,y)\|_F^2\right) \\&~~~~~~~~~~~ + 36(n+2)\left(\sigma_{2,Q}^2 + \|\nabla^2_{xy}q(x,y)\|_F^2\right) \\& ~~~~~~~~~~~ + \frac{30\mu^2}{\eta^2}n(m+2)\left(\sigma_{2,Q}^2 + \|\nabla^2_{yy}q(x,y )\|_F^2\right)\bigg]\|\theta\|^2,
			\end{align*}
			where $\tilde \nabla^2_{xy} Q_{\eta,\mu}(x,y,\{\zeta_i\}_{i=1}^N)$ is defined in \eqref{stochHessxy_minibatch}.
		\end{itemize}
	\end{proposition}
	\begin{proof}
	    See Section \ref{SMProof:1} of the Supplement.
	\end{proof}

	\subsection{Notation and Definitions} The vertical concatenation of two vectors $x_1\in\mathbb{R}^{d_1}$ and $x_2\in\mathbb{R}^{d_2}$ (i.e., $[x_1^\top,x_2^\top]^\top$) which is a vector in $\mathbb{R}^{d_1+d_2}$ would be denoted by $(x_1,x_2)$. We reserve $I$ for the identity matrix. For a given (stochastic) function $q$, we use $q\in \mathcal{C}^i(S;L_{i,q})$ to denote its smoothness over the set $S$, where $L_{i,q}$ represents the (expected) Lipschitz constant of $i$-th derivative of $q$ ($i=0$ corresponds to the function itself). We also use $\lambda_q$ to show the strong convexity modulus of function $q$. For a stochastic function $Q(z,\zeta)$, with a random parameter $\zeta$, we use $\sigma_{i,Q}^2$ to show global upper bounds on the variance of its $i$-th derivative.

	\section{Nearly Optimal Zeroth-Order Bilevel Algorithm}\label{near_opt}
	In this section, we present our proposed algorithm to solve the BLP problem in \eqref{main_prob_st} that only uses zeroth-order information of both upper and lower objective functions. We fisrt define a smoothed variant of \eqref{main_prob_st} as follows:
	\begin{align} \label{main_prob_zst}
		&\min_{x \in X} \Big\{\psi_{\eta,\mu}(x):= f_{\eta_1,\mu_1}(x,y_{\eta_2,\mu_2}^*(x))=\bbe[F(x+\eta_1 u,y_{\eta_2,\mu_2}^*(x)+\mu_1 v,\xi)]\Big\} \nn \\
		& \text{s.t.} \ \  y_{\eta_2,\mu_2}^*(x) = \argmin_{y \in \bbr^m} \Big\{g_{\eta_2,\mu_2}(x,y)= \bbe[G(x+\eta_2 u,y+\mu_2 v,\zeta)]\Big\},
	\end{align}
	where $\mathbf{\eta}=(\eta_1,\eta_2)$ and $\mathbf{\mu}=(\nu_1,\nu_2)$ are the smoothing parameters, $u \sim\mathcal{N}(0,I_n)$, $v\sim\mathcal{N}(0,I_m)$ are the standard normal vectors. Similar to Lemma~\ref{def_Mxy}, we can calculate the gradient of $\psi_{\eta,\mu}$ as follows:
	\begin{align}\label{def_smooth_hypergrad}
		\nabla \psi_{\eta,\mu}(\bar x)& =~ \nabla_x f_{\eta_1,\mu_1}(\bar x, y^*_{\eta_2,\mu_2}(\bar x))\nonumber \\
		& - \nabla_{xy}^2 g_{\eta_2,\mu_2}(\bar x, y^*_{\eta_2,\mu_2}(\bar x))\left[\nabla_{yy}^2 g_{\eta_2,\mu_2}(\bar x, y^*_{\eta_2,\mu_2}(\bar x))\right]^{-1}  \nabla_y f_{\eta_1,\mu_1}(\bar x, y^*_{\eta_2,\mu_2}(\bar x)).
	\end{align}
	We point out that $\psi_{\eta,\mu}$ and $\nabla \psi_{\eta,\mu}$ are introduced to only later simplify the notations in our convergence analysis.
	
	An essential step in designing most algorithms for solving BLP problems is estimating the gradient of the upper-level objective function, which involves computing the inverse of a Hessian matrix within the gradient of the function $\psi$, a process that is computationally expensive. Two main approaches have been proposed to address this computational challenge: the use of Neumann series (e.g., \cite{GhadWang18,Ji2020BilevelOC}) and stochastic gradient descent (SGD) methods (e.g., \cite{ArbelM22,chen2023optimal,dagrou2022a}). In this subsection, we provide a zeroth-order variant of the SGD algorithm from \cite{AghaGhad25} to approximate the Hessian inverse product in \eqref{def_smooth_hypergrad}. This algorithm indeed applies a biased SGD to the quadratic problem 
	\[
	\bar z = \argmin_z ~J(z) := \frac{1}{2} z^\top \nabla_{yy}^2 g_{\eta_2,\mu_2}(\bar x, \bar y) z -  \nabla_y f_{\eta_1,\mu_1}(\bar x, \bar y)^\top z,
	\]
	whose optimal solution is given by
	\begin{equation}\label{zbar:def}
		\bar z := \left[\nabla_{yy}^2 g_{\eta_2,\mu_2}(\bar x, \bar y)\right]^{-1}  \nabla_y f_{\eta_1,\mu_1}(\bar x, \bar y),
	\end{equation}
	for given $\bar x$ and $\bar y$. 
	\begin{algorithm} [H]
		\caption{The Stochastic Zeroth-order Hessian Inverse Approximation \cite{AghaGhad25}}
		\label{Hinv:alg1}
		\begin{algorithmic}
			
			\STATE \textbf{Input}:
			$\bar x \in X$, $\bar y \in \bbr^m$, smoothing parameters $\eta_2,\mu_1,\mu_2$, and maximum iterations $T$.
			\STATE Initialize $z_0$
			\STATE {\bf For $\tau=0,\ldots$, $T-1$:}
			
			{\addtolength{\leftskip}{0.2in}
				
				\STATE {Draw i.i.d random vectors $\zeta_{\tau}, \xi_\tau$ and $u_{\tau},v_\tau,u'_{\tau},v'_\tau$ independently from a standard Gaussian distribution, and calculate $\tilde \nabla J(z_\tau)$ via
					\begin{align}\notag 
						   \Big( v_\tau v_\tau^\top - &I_m\Big)\! \left[\frac{G(\bar x+\eta_2 u_\tau, \bar y+\mu_2 v_\tau, \zeta) + G(\bar x-\eta_2 u_\tau,\bar y-\mu_2 v_\tau, \zeta)-2G(\bar x,\bar y,\zeta)}{2\mu_2^2} \right]z_\tau\\&   -v_\tau'\!\left[  \frac{F(\bar x, \bar y+\mu_1 v_\tau', \xi) - F(\bar x, \bar y,\xi)}{\mu_1} \ \right],\label{gradJ:def}
					\end{align}
				}\\
				Set $$z_{\tau+1} = z_\tau - \beta \tilde \nabla J(z_\tau).$$
				\\
			}
			
			{\bf End}
			\STATE Output: $H_{\eta_2,\mu_1,\mu_2} = z_T$.
		\end{algorithmic}
	\end{algorithm}

	In the next result, we provide convergence analysis of the above algorithm. 
	\begin{lemma}\label{sgd:hessInv}
		Let $\bar z$ be defined in \eqref{zbar:def} and
		$z_T$ be the output of Algorithm~\ref{Hinv:alg1} with the choice of $\mu_1 \le 1/m$, $\eta_2=\mu_2 \le 1/\sqrt{m+n}$, $\gamma = {\cal O} \left(\frac{\epsilon}{m(m+n)^2}\right)$, and $T = \frac{m(m+n)^2}{\epsilon}\log(1/\epsilon)$. Then, 
		\begin{align*}
			\EE \left[\| z_T - \bar z\|^2 \right] = {\cal O}(\epsilon).
		\end{align*}
	\end{lemma}
\begin{proof}
    See Section \ref{sec:2proof1} of the Supplement. 
\end{proof}

	We should point out that by choosing smaller stepsizes and larger number of iterations than the ones used in \cite{AghaGhad25}, we could reduce the error in estimating the Hessian inverse vector product  to ${\cal O}(\epsilon)$ which will be useful later to obtain an overall finer complexity in terms of the target accuracy $\epsilon$.
	
	We are now ready to present our first algorithm for solving the BLP problem in \eqref{main_prob_st}, which is a variant of the one in \cite{AghaGhad25} with the major difference that we use mini-batch of samples to estimate all first- and second-order derivatives. The algorithm uses a double-loop framework in which the inner loop applies a SGD-type method to the lower level objective function while the outer-loop applies a similar SGD-type method to the upper level objective function (see e.g., \cite{GhadWang18,chen2021closing,franceschi2018bilevel} for double-loop algorithmic frameworks for BLP).
	
	\begin{algorithm}
		\caption{The Zeroth-order Mini-batch Double-loop Stochastic Bilevel Approximation Method}
		\label{alg_ZBSA}
		\begin{algorithmic}
			
			\STATE Input:
			Smoothing parameters $\eta_1,\eta_2,\mu_1,\mu_2$, a maximum number of iteration $N \ge 1$, an initial solution $x_0 \in X$, $y_0 \in \bbr^m$ nonnegative sequences $\{\alpha_k\}_{k \ge 0}$, $\{\beta_t\}_{t \ge 0}$, and integer sequences $\{t_k\}_{k \ge 0}$ and $\{b_k\}_{k \ge 0}$.
			
			
			{\bf For $k=0,1,\ldots$, $N$:}
			
			\vgap
			
			{\addtolength{\leftskip}{0.2in}

				{\bf For $t=0,1,\ldots, t_k-1$:}
				
			}
			\vgap
			{\color{black}
				{\addtolength{\leftskip}{0.4in}
					
					\STATE Compute partial gradient approximations of $G_{0,\mu}$ w.r.t $y$:
					\beq
					\tilde \nabla_y G_{0,\mu_2}^t  = \left[ \frac{G(x_k, y_t+\mu_2 v^{(1)}_t, \zeta^{(1)}_t) - G(x_k, y_t,\zeta^{(1)}_t)}{\mu_2}\right] v^{(1)}_t\label{sgrady}
					\eeq
					where $\zeta^{(1)}_t$ and $v^{(1)}_t$ are i.i.d samples from $\zeta$ and Gaussian distribution, respectively. Set
					\beq \label{def_yt}
					y_{t+1} = y_t- \beta_t \tilde \nabla_y G_{0,\mu_2}^t.
					\eeq
					
				}
				
				{\addtolength{\leftskip}{0.2in}
					{\bf End}
					
					\STATE  Set $\bar y_k=y_{t_k}$ and compute the partial gradient approximations:
					\begin{align}
						\tilde \nabla_x F_{\eta_1,\mu_1}^k  &= \frac{1}{s_k} \sum_{i=1}^{s_k} \tilde \nabla_x F_{\eta_1,\mu_1}^{k,i} ,\label{sgradfx0}\\
						\tilde \nabla_x F_{\eta_1,\mu_1}^{k,i}  &= \left[ \frac{F(x_k+\eta_1 u^{(2)}_{k,i}, \bar y_k, \xi_{k,i}) - F(x_k, \bar y_k,\xi_{k,i})}{\eta_1}\right] u^{(2)}_{k,i},\label{sgradfx}\\
						\tilde \nabla^2_{xy} G_{\eta_2,\mu_2}^k &= \frac{1}{s_k} \sum_{i=1}^{s_k} G_{\eta_2,\mu_2}^{k,i} u^{(3)}_{k,i} {v^{(3)}_{k,i}}^\top  \label{Hessian_xy1},
					\end{align}
					where
					\begin{align*}
					&G_{\eta_2,\mu_2}^{k,i} =\\ & \frac{G(x_k\!+\!\eta_2 u^{(3)}_{k,i}, \bar y_k\!+\!\mu_2 v^{(3)}_{k,i}, \zeta^{(3)}_{k,i}) + G(x_k\!-\!\eta_2 u^{(3)}_{k,i}, y_k\!-\!\mu_2 v^{(3)}_{k,i}, \zeta^{(3)}_{k,i})- 2 G(x_k,\bar y_k, \zeta^{(3)}_{k,i})}{2 \eta_2 \mu_2},
					\end{align*}
					and $\zeta^{(j)}_{k,i}$, $\xi_{k,i}$, and $(u^{(j)}_k,v^{(j)}_{k,i})$ are i.i.d samples from $\zeta$, $\xi$, and Gaussian distribution, respectively.
					Call Algorithm~\ref{Hinv:alg1} to compute $\tilde H^k_{\bar \eta,\bar \mu}$ with the inputs $(x_k,\bar y_k, \eta_2, \mu_1, \mu_2, T=b_k)$ and set
					\beqa
					\tilde \nabla \psi_{\bar \eta,\bar \mu}^k \equiv \tilde \nabla_x f_{\bar \eta,\bar \mu}(x_k,\bar y_k,\w_k) &=& \tilde \nabla_x F_{\eta_1,\mu_1}^k  - \tilde \nabla^2_{xy} G_{\eta_2,\mu_2}^k 
					\tilde H^k_{\bar \eta,\bar \mu},\label{grad_f_st}
					\eeqa
					where $(\bar \eta,\bar \mu) = (\eta_1,\eta_2,\mu_1,\mu_2)$ and $\w_k =(\xi_{[s_k]},\zeta^{(1)}_{[t_k]}, \zeta^{(2)}_k, \zeta^{(3)}_{[s_k]}, u^{(1)}_k, v^{(1)}_k, u^{(2)}_{[s_k]}, v^{(2)}_{[s_k]}, u^{(3)}_{[s_k]}, v^{(3)}_{[s_k]})$.
					Set
					\beq \label{def_xk_st}
					x_{k+1} = \arg\min_{x \in X} \left\{\langle \tilde \nabla \psi_{\bar \eta,\bar \mu}^k,x -x_k \rangle + \frac{1}{2 \alpha_k}\|x-x_k\|^2 \right\},
					\eeq
					
				}
			}
			
			{\bf End}
		\end{algorithmic}
	\end{algorithm}
	
	In the next result, we provide a convergence analysis of 
	the inner loop in Algorithm~\ref{alg_ZBSA} which uses a similar proof strategy \cite{AghaGhad25} with different choice of parameters to establish an improved sample complexity.
	
	\begin{lemma}\label{lemma_sgd}
		Let $\{y_t\}_{t=0}^{t_k}$ be the sequence generated at the $k$-th iteration of Algorithm~\ref{alg_ZBSA}. If
		\beq
		\beta_t = {\cal O}\left(\frac{\epsilon}{m}\right), \ \ t_k ={\cal O}\left(\frac{m}{\epsilon} \log \left(\frac{m}{\epsilon}\right) \right), \ \ \eta_2 = \mu_2 \le \sqrt{\min \left(1, \frac{\epsilon}{n}, \frac{1}{m^3} \right)}.\label{alpha_beta_st}
		\eeq
		we have
		\[
		\EE \|y_{t_k} - y_{\eta_2,\mu_2}^*(x_k) \|^2  = {\cal O} (\epsilon).
		\]
	\end{lemma}
\begin{proof}
    See Section \ref{sec:2proof2} of the Supplement. 
\end{proof}

	We are now ready to present the main convergence results of Algorithm~\ref{alg_ZBSA}.
	
	\begin{theorem}\label{theom_main_bsa}
		Suppose that $\{\bar y_k, x_k \}_{k \ge 0}$ is generated by Algorithm~\ref{alg_ZBSA}, Assumptions~\ref{fg_assumption} and ~\ref{stochastic_assumption} hold, $t_k$ and $\beta_t$ are chosen according to  \eqref{alpha_beta_st}, and 
		\begin{align}
			\alpha_k &\le \frac{1}{5L_{1,\psi}}, \qquad \qquad \qquad \qquad \qquad \quad \quad s_k = \frac{\max(24 (n+2), \sqrt{nm})}{\epsilon}\label{alpha_sk},\\
			\eta_1=\mu_1 &= \min \left\{\frac{1}{(n+m)^2}, \sqrt{\frac{\epsilon}{(n+m)^3}}\right\}, \qquad \eta_2=\mu_2 = \min \left\{\frac{1}{n+m},\sqrt{\frac{\epsilon}{n+m}} \right\}.\label{smooth_params}
		\end{align}

			If $\psi$ is possibly nonconvex, $X=\bbr^n$ (for simplicity)\footnote{We should mention that our results hold for the general case of $X$, however, we need to define a notion of generalized gradient in the constraint case such as gradient mapping, which has been removed to simplify our presentation.}, we have
			\begin{align}\label{bg_nocvx}
				\bbe\|\nabla \psi(x_R)\|^2 \le \frac{ \psi(x_0) -\psi^* +  {\cal O }(\epsilon) \sum_{k=1}^N (\alpha_k+\alpha_k^2)}{\sum_{k=0}^{N-1} \alpha_k},
			\end{align}
			where the expectation is also taken with respect to the integer random variable $R$ whose probability mass function is given by
			\beq
			Prob(R=k) = \frac{\alpha_k}{\sum_{\tau=0}^{N-1}\alpha_{\tau}} \qquad k=0,1,\ldots, N-1.\label{def_R}
			\eeq
		\end{theorem}

		\begin{proof}
			Noting \eqref{def_xk_st} and the smoothness of $\psi$ due to Lemma~\ref{grad_f_error}.(c), we have
			\begin{align}
				\psi(x_{k+1}) &\le \psi(x_k) + \langle \nabla \psi(x_k),x_{k+1}-x_k \rangle +\frac{L_{1,\psi}}{2}\|x_{k+1}-x_k\|^2 \nn \\
				&= \psi(x_k) - \alpha_k \langle \nabla \psi(x_k), \tilde \nabla \psi_{\bar \eta,\bar \mu}^k \rangle +\frac{L_{1,\psi}\alpha_k^2}{2}\|\tilde \nabla \psi_{\bar \eta,\bar \mu}^k\|^2. \label{psi_smooth1}
			\end{align}
			Let $\tilde \Delta_k \equiv \tilde \nabla \psi_{\bar \eta,\bar \mu}^k - \nabla \psi(x_k)$, and denote by $\bbe_{|k}$ the conditional expectation given all random sequences generated up to the $k$-th iteration of Algorithm~\ref{alg_ZBSA}, then
			\begin{align}
				\bbe_{|k}[\psi(x_{k+1})] &\le \psi(x_k) -\alpha_k  \langle \nabla \psi(x_k) , \bbe_{|k} [\tilde \nabla \psi_{\bar \eta,\bar \mu}^k] \rangle
				+\frac{L_{1,\psi} \alpha_k^2}{2}\bbe_{|k}\|\tilde \nabla \psi_{\bar \eta,\bar \mu}^k\|^2 \nn \\
				&= \psi(x_k) -\alpha_k \|\nabla \psi(x_k)\|^2 -\alpha_k \langle \nabla \psi(x_k), \bbe_{|k}[\tilde \Delta_k] \rangle
				+\frac{L_{1,\psi} \alpha_k^2}{2}\bbe_{|k}\|\tilde \nabla \psi_{\bar \eta,\bar \mu}^k\|^2 \nn \\
				&\le \psi(x_k) -\frac{\alpha_k}{2}(1-2L_{1,\psi} \alpha_k)\|\nabla \psi(x_k)\|^2 +\frac{\alpha_k}{2} \left\|\bbe_{|k}[\tilde \Delta_k]\right\|^2
				\nn \\
				&\quad +L_{1,\psi} \alpha_k^2\bbe_{|k}\|\tilde \Delta_k\|^2.\label{smooth_psi}
			\end{align}
			Now, we bound the last two terms on the right-hand side of the above inequality. Denoting $\tilde \delta_k =\tilde \nabla \psi_{\bar \eta,\bar \mu}^k - \bar \nabla (f_{\eta_1,  \mu_1}, g_{\eta_2,\mu_2})(x_k,\bar y_k)$,
			$\delta_k = \bar \nabla (f_{\eta_1,  \mu_1}, g_{\eta_2,\mu_2})(x_k,\bar y_k)-\nabla \psi_{\bar \eta,\bar \mu} (x_k)$, and $\Delta_k = \nabla \psi_{\bar \eta,\bar \mu} (x_k)-\nabla \psi(x_k)$, 
			we have $\tilde \Delta_k = \tilde \delta_k+ \delta_k +\Delta_k$, where $\bar \nabla (f,g)(\cdot, \cdot)$ is defined in \eqref{grad_f}. Using the superscript convention $h^k \equiv h(x_k,\bar y_k)$ for any function $h$, we have
			\begin{align*}
				\bbe_{|k}[\tilde \delta_k] &=   \bbe_{|k}[\tilde \nabla_x F_{\eta_1,\mu_1}^k - \nabla_x f^k_{\eta_1,\mu_1}] \\
				&\quad + \bbe_{|k}\left[\nabla^2_{xy} g^k_{\eta_2,\mu_2}[\nabla^2_{yy} g^k_{\eta_2,\mu_2}]^{-1}\nabla_y f^k_{\eta_1,\mu_1}- \tilde \nabla^2_{xy} G_{\eta_2,\mu_2}^k 
				\tilde H^k_{\bar \eta,\bar \mu}\right] \\
				&= \nabla^2_{xy} g^k_{\eta_2,\mu_2}\left([\nabla^2_{yy} g^k_{\eta_2,\mu_2}]^{-1}\nabla_y f^k_{\eta_1,\mu_1}- \bbe_{|k}\left[\tilde H^k_{\bar \eta,\bar \mu}\right]\right)
				\\&~~~ +\bbe_{|k}\left[\left(\nabla^2_{xy} g^k_{\eta_2,\mu_2} - \tilde \nabla^2_{xy} G_{\eta_2,\mu_2}^k\right)
				\tilde H^k_{\bar \eta,\bar \mu}\right] \\
				&= \nabla^2_{xy} g^k_{\eta_2,\mu_2}\left([\nabla^2_{yy} g^k_{\eta_2,\mu_2}]^{-1}\nabla_y f^k_{\eta_1,\mu_1}- \bbe_{|k}\left[\tilde H^k_{\bar \eta,\bar \mu}\right]\right),
			\end{align*}
			which in the view of Lemma~\ref{sgd:hessInv}, implies that 
			\beq
			\left\|\bbe_{|k}[\tilde \delta_k] \right\|^2 \le \left\|\nabla^2_{xy} g^k_{\eta_2,\mu_2} \right\|^2 \left\|[\nabla^2_{yy} g^k_{\eta_2,\mu_2}]^{-1}\nabla_y f^k_{\eta_1,\mu_1} \bbe_{|k}-\left[\tilde H^k_{\bar \eta,\bar \mu}\right] \right\|^2 \le {\cal O}(\epsilon).\label{def_tildeAk} 
			\eeq
			Moreover, it is true that
			\begin{align}
				\bbe_{|k}[\|\tilde \delta_k\|^2] &\le 
				3\bbe_{|k} \left\|\tilde \nabla_x F_{\eta_1,\mu_1}^k - \nabla_x f^k_{\eta_1,\mu_1}\right\|^2 \nn\\
				&~~~+3 \left\|\nabla^2_{xy} g^k_{\eta_2,\mu_2} \right\|^2 \bbe_{|k}\left\|\left[\nabla^2_{yy} g^k_{\eta_2,\mu_2}\right]^{-1}\nabla_y f^k_{\eta_1,\mu_1} - \tilde H^k_{\bar \eta,\bar \mu}\right\|^2 \nn \\&~~~+3\bbe_{|k}\left\|\left(\nabla^2_{xy} g^k_{\eta_2,\mu_2}- \tilde \nabla^2_{xy} G_{\eta_2,\mu_2}^k \right) \tilde H^k_{\bar \eta,\bar \mu}\right\|^2.
				\label{delta_expec}
			\end{align}
			Assuming $\eta_1=\mu_1\le 1/(m+n)^2$, by Proposition~\ref{propNestApprox_stch}.(a) we have
			\begin{align*}
				\bbe_{|k}\! \left\|\tilde \nabla_x F_{\eta_1,\mu_1}^k \!-\! \nabla_x f^k_{\eta_1,\mu_1}\right\|^2  &
				\le \frac{L_{1,Q}^2 \eta_1^2}{s_k}\left((n+6)^3+n(m+4)^2 \right) \\ &~~+  \frac{4(n+2)}{s_k}\left(\sigma_{1,Q}^2\!+\! \|\nabla_y  f(x, y)\|^2\right) +  \frac{4(n+2)}{s_k}\|\nabla  \psi(x_k)\|^2\\
				&\le {\cal O } \left(\frac{n}{s_k} \right) + \frac{4(n+2)}{s_k}\|\nabla  \psi(x_k)\|^2.
			\end{align*}
			Moreover, by Lemma~\ref{sgd:hessInv}:
			\begin{align*}
				\bbe_{|k}\left\|\left[\nabla^2_{yy} g^k_{\eta_2,\mu_2}\right]^{-1}\nabla_y f^k_{\eta_1,\mu_1} - \tilde H^k_{\bar \eta,\bar \mu}\right\|^2  &\le {\cal O}(\epsilon),\\
				\bbe_{|k}\left\| \tilde H^k_{\bar \eta,\bar \mu}\right\|^2  &\le {\cal O}(\epsilon) + \left\|\left[\nabla^2_{yy} g^k_{\eta_2,\mu_2}\right]^{-1}\nabla_y f^k_{\eta_1,\mu_1} \right\|^2 = {\cal O}(1).
			\end{align*}
			In addition, assuming $\mu_2=\eta_2 \le 1/(n+m)$ and by Proposition~\ref{propNestApprox_stch}.(b) we get
			\begin{align*}
				& \bbe_{|k}\left\|\left(\nabla^2_{xy} g^k_{\eta_2,\mu_2}- \tilde \nabla^2_{xy} G_{\eta_2,\mu_2}^k \right) \tilde H^k_{\bar \eta,\bar \mu}\right\|^2 
				\bbe_{|k}\left[\left\|\tilde H^k_{\bar \eta,\bar \mu_2}\right\|^2\right]^{-1} \\
				&\le \frac{8L_{2,G}^2}{s_k}\left[ \eta_2^2(n+8)^4 \!+\! 2\mu_2^2 n(m+12)^3\right]+\frac{1}{s_k} \bigg[ 6(n+4)(n+2)\left(\sigma_{2,G}^2 \!+\! \|\nabla^2_{xx} g(x,y)\|_F^2\right) \\
				&~~~+ 36(n+2)\left(\sigma_{2,G}^2 + \|\nabla^2_{xy}g(x,y)\|_F^2\right)+ 30 n(m+2)\left(\sigma_{2,G}^2 + \|\nabla^2_{yy}g(x,y )\|_F^2\right)\bigg] \\&={\cal O} \left(\frac{n+\sqrt{n m}}{s_k} \right).
			\end{align*}
			Combining the above inequalities yields 
			\[
			\bbe_{|k}\left\|\tilde \delta_k \right\|^2 \le {\cal O } \left(\frac{n+\sqrt{n m}}{s_k} \right) + {\cal O }(\epsilon) +\frac{12(n+2)}{s_k}\|\nabla  \psi(x_k)\|^2.
			\]
			In addition, in the view of Proposition~\ref{zeroth-order bilevel}, Lemma~\ref{grad_f_error}.(a), Lemma~\ref{lemma_sgd}, and assuming $\mu_2 =\eta_2 \le \sqrt{\epsilon/(n+m)}, \mu_1=\eta_1 = \sqrt{\epsilon/(n+m)^3}$, we also have
			\begin{align}
				\bbe_{|k}\|\Delta_k\|^2 \le {\cal O}(\epsilon),\qquad 
				\bbe_{|k}\|\delta^k\|^2 \le C_1^2 \EE_{|k} \|y_{t_k} - y_{\eta_2,\mu_2}^*(x_k) \|^2  \le {\cal O}(\epsilon).\label{def_A_delta}
			\end{align}
			Combining the above inequalities, we obtain
			\begin{align}
				\left\|\bbe_{|k}[\tilde \Delta_k]\right\|^2 &\le 3 \left(\left\|\bbe_{|k}\left[\tilde \delta_k\right]\right\|^2 + \bbe_{|k}\|\delta^k\|^2 + \bbe_{|k}\|\Delta_k\|^2\right) \le {\cal O }(\epsilon),\\
				\bbe_{|k}\|\tilde \Delta_k\|^2 &\le 3 \left(\bbe_{|k}\left\|\tilde \delta_k\right\|^2 + \bbe_{|k}\|\delta^k\|^2 + \bbe_{|k}\|\Delta_k\|^2\right)\nn\\
				&\le {\cal O }(\epsilon) + {\cal O } \left(\frac{n+\sqrt{n m}}{s_k} \right) +\frac{12(n+2)}{s_k}\|\nabla  \psi(x_k)\|^2,
			\end{align}
			which together with \eqref{smooth_psi}, imply that
			\begin{align*}
				\bbe_{|k}[\psi(x_{k+1})] &\le \psi(x_k) -\frac{\alpha_k}{2}\left(1-2L_{1,\psi} \alpha_k \left[1+\frac{24(n+2)}{s_k} \right] \right)\|\nabla \psi(x_k)\|^2 \\
				& ~~~+(\alpha_k+\alpha_k^2) {\cal O }(\epsilon)
				\nn + \alpha_k^2 {\cal O } \left(\frac{n+\sqrt{n m}}{s_k} \right).
			\end{align*}
			Choosing the batch size as $s_k = \max(24 (n+2), \sqrt{nm})/\epsilon$ and assuming $\epsilon \le 1$, we have
			\begin{align*}
				\bbe_{|k}[\psi(x_{k+1})] &\le \psi(x_k) -\frac{\alpha_k}{2}\left(1-4L_{1,\psi} \alpha_k \right)\|\nabla \psi(x_k)\|^2 +  (\alpha_k+\alpha_k^2) {\cal O }(\epsilon).
			\end{align*}
			Summing up the above inequalities, re-arranging the terms, and noting the fact that $\alpha_k \le 1/(5L_{1,\psi})$, we obtain
			\[
			\sum_{k=0}^{N-1} \alpha_k \|\nabla \psi(x_k)\|^2 \le \psi(x_0) -\psi^* +  {\cal O }(\epsilon) \sum_{k=1}^N (\alpha_k+\alpha_k^2).
			\]
			Divining both sides of the above inequality by $\sum_{k=1}^N \alpha_k$, and noting that
			\[
			\bbe[\|\nabla \psi(x_R)\|^2] = \sum_{k=0}^{N-1} \frac{\alpha_k}{\sum_{\tau=0}^{N-1}\alpha_{\tau}}\bbe[\|\nabla \psi(x_k)\|^2],
			\]
			we obtain \eqref{bg_nocvx}.
			
		\end{proof}

		In the next result, we specialize the rate of convergence of Algorithm~\ref{alg_ZBSA}.
		
		\begin{corollary} \label{lemma_main_bsa}
			Suppose that $\{x_k ,\bar y_k, \}_{k \ge 0}$ is generated by Algorithm~\ref{alg_ZBSA}, Assumptions~\ref{fg_assumption} and \ref{stochastic_assumption} hold. Also, assume that $t_k$ and $\beta_k$ are set to \eqnok{alpha_beta_st}, respectively, and smoothing parameters are set to \eqnok{smooth_params}. Choosing
			\beq\label{alpha_N}
			\alpha_k= \frac{1}{5L_{1,\psi}}, \qquad N = {\cal O} \left(\frac{1}{\epsilon}\right).
			\eeq
			we have
			\beq\label{bg_nocvx1_st}
			\bbe\left[\|\nabla \psi(x_R)\|^2\right] \le {\cal O} (\epsilon)
			\eeq
			where the expectation is also taken with respect to the integer random variable $R$ uniformly distributed over $\{0,1,\ldots, N-1\}$. Furthermore, the total number of noisy evaluations of $f$ and $g$ to obtain such a solution is boudned by 
			\beq\label{total_complx}
			{\cal O} \left(\frac{m(m+n)^2 \log(1/\epsilon)}{\epsilon^2 } \right).
			\eeq
		\end{corollary}

		\begin{proof}
			First, note that by \eqref{alpha_N} and \eqref{bg_nocvx1_st}, we have 
			\begin{align}
				\bbe[\|\nabla \psi(x_R)\|^2] \le \frac{5L_{1,\psi}[\psi(x_0) -\psi^*] }{N} +  {\cal O }(\epsilon) = {\cal O }(\epsilon).
			\end{align}
			Furthermore, the total sample complexity of Algorithm~\ref{alg_ZBSA} is bounded by
			\begin{align}
				\sum_{k=1}^N (s_k +t_k+b_k) &=  \sum_{k=1}^N \Bigg[\max(n, \sqrt{mn}) {\cal O}\left(\frac{1}{\epsilon} \right)+ {\cal O}\left(\frac{m}{\epsilon} \log \left(\frac{m}{\epsilon}\right) \right) \\
				&+ m(m+n)^2 {\cal O} \left(\frac{1}{\epsilon}\log(1/\epsilon)\right) \Bigg] \nn \\
				&= {\cal O}\left(\frac{1}{\epsilon^2} \right) \left[\max(n, \sqrt{mn})  + m \log(1/\epsilon) + m(m+n)^2 \log(1/\epsilon) \right]
			\end{align}

		\end{proof}

		Notice that the complexity bound in \eqref{total_complx} is better than the bound of
		${\cal O}\left(\frac{(n+m)^4}{\epsilon^3 }\log \left(\frac{n+m}{\epsilon}\right)\right)$ presented in \cite{AghaGhad25} in terms of dependence on both problem dimension and target accuracy. Indeed, it is optimal (up to a logarithmic factor) in terms of the dependence on $\epsilon$, while cubically depends on the problem dimension which is worse than the optimal linear dependency in the context of stochastic zeroth-order optimization. In the next section, we address this issue by using our zeroth-order estimates in the framework of a fully first-order bilevel optimization algorithms thereby avoiding the need to estimate second-order derivatives of $g$. 
		
\section{Optimal Zeroth-order Bilevel Algorithm}\label{full_opt}
In the previous section, we introduced a fully zeroth-order method based on Gaussian smoothing, leveraging Hessian-inverse approximation techniques. However, when only zeroth-order queries are available, the sample complexity of these techniques scales cubically with the problem dimension, which can be prohibitive in large-scale applications. Motivated by first-order approaches in the literature \cite{kwon2023fully}, in this section we present a zeroth-order method whose sample complexity scales linearly with dimension while preserving the optimal dependence on target accuracy.

The key idea here is a formulation which penalizes the sub-optimality in the lower level:
\begin{align*}
    \mathcal{L}(x,y) = f(x,y) + \lambda(g(x,y) - g^*(x)),
\end{align*}
where $g^*(x) := \min_z g(x,z)$. It has been shown in \cite{kwon2023fully} that for a sufficiently large $\lambda \ge 4L_{1,f} / \lambda_g$, one can obtain an $\mathcal{O}(1/\lambda)$-approximate surrogate of $\psi(x)$.
\begin{lemma} \cite[Lemma 3.1]{kwon2023fully} \label{lemma:fully_my_lemma}
    Let Assumption \ref{fg_assumption} holds, and define
    \begin{align*}
        \mathcal{L}^*(x) := \min_y \mathcal{L}(x,y) = \min_y \max_z f(x,y) + \lambda(g(x,y) - g(x,z)).
    \end{align*} 
    For all $\lambda > L_{1,f}/\lambda_g$ and any $x \in X$, $\nabla \mathcal{L}^*(x)$ satisfies
    \begin{align*}
        \|\nabla \mathcal{L}^*(x) - \nabla \psi(x)\| &\le L_\lambda /\lambda,
    \end{align*}
    where $L_\lambda = \frac{4 L_{0,f} L_{1,g}}{\lambda_g^2} \left( L_{1,f} + \frac{2 L_{0,f} L_{2,g}}{\lambda_g} \right)$.
\end{lemma}
Thus, with $\lambda = \mathcal{O} (\epsilon^{-1})$, we can focus on minimizing $\mathcal{L}^*(x)$ instead of $\psi(x)$ directly. To see how this surrogate formulation yields an algorithm that relies solely on first-order derivatives (thereby improving $\mathcal{O}(n+m)$-dependency), note that our surrogate objective can be rewritten as:
\begin{align*}
    \min_x \mathcal{L}^*(x) &= \min_x  (\min_y f(x,y) + \lambda g(x,y) - \min_z  \lambda g(x,z)) \\
    &= \min_x f(x,y^*(x)) + \lambda g(x,y^*(x)) - \lambda g(x,z^*(x)),
\end{align*}
where in this notation we define
\begin{align}
    y^*(x) &:= \arg\min_y \lambda^{-1} f(x,y) + g(x,y),\label{y_first} \\
    z^*(x) &:= \arg\min_z g(x,z).\label{z_first}
\end{align}
Moreover, $\frac{1}{\lambda} f(x,y) + g(x,y)$ is $(3\lambda_g/4)$-strongly convex with $\lambda > 4L_{1,f}/\lambda_g$, and hence, by Danskin's Theorem, we can verify that the gradient of the surrogate objective consists only of first-order derivatives:
\begin{align*}
    \nabla \mathcal{L}^*(x) = \nabla_x f(x,y^*(x)) + \lambda \nabla_x g(x,y^*(x)) - \lambda \nabla_x g(x,z^*(x)).
\end{align*}
To see how this expression approximates $\nabla \psi(x)$, we state the following lemma.
\begin{lemma}\label{lemma:yz:diffs}
    Let $\lambda \ge 4L_{1,f} / \lambda_g$. Then for any $x \in X$, we have
    \begin{align*}
        \|y^*(x) - z^*(x)\| & \le \frac{2L_{0,f}}{\lambda_g \lambda}, \\
        y^*(x) - z^*(x) &= -\frac{1}{\lambda} \nabla_{yy}^2 g(\bar{x},z^*(x))^{-1} \nabla_y f(x,y^*(x)) + \mathcal{O} \left( \frac{L_{2,g} L_{0,f}^2}{\lambda_g^3 \lambda^2} \right).
    \end{align*}
\end{lemma}
\begin{proof}
    See Section \ref{sec:3proof1} of the Supplement.
\end{proof}
Thus, $\lambda (\nabla_x g(x,y^*(x)) - \nabla_x g(x,z^*(x)))$ is an approximation of the second-order part in $\nabla \psi(x)$ roughly through finite differentiation with a large enough $\lambda > 0$. Remaining challenge is that, as in the previous section, we clearly cannot access the exact $y^*(x), z^*(x)$; which we resolved using double-loop iterations (it can also be resolved using two-timescale stepsizes). 

Our fully zeroth-order algorithm based on this formulation, motivated by  \cite{kwon2023fully}, is described in Algorithm \ref{alg_ZBSA2}. In the inner loop of this algorithm, we apply two SGD-type of updates to get approximate solutions of \eqnok{y_first} and \eqnok{z_first}, which is similar to that of Algorithm~\ref{alg_ZBSA} in which we just have one updating sequence. Also, in the outer loop, we apply a projected SGD-type update to the sequence $\{x_k\}$ for minimizing the modified upper objective function, which unlike the one in Algorithm~\ref{alg_ZBSA}, does not require any estimate of Hessian matrices or Hessian inverse-vector product.

\begin{algorithm}
	\caption{The Zeroth-order Mini-batch Double-loop Stochastic Bilevel Approximation First-Order Method}
	\label{alg_ZBSA2}
	\begin{algorithmic}

    \STATE Input:
    Smoothing parameters $\eta, \mu > 0$, a penalty parameter $\lambda > 0$, a maximum number of iteration $N \ge 1$, an initial solution $x_0 \in X$, $y_0, z_{0} \in \bbr^m$, nonnegative sequences $\{\alpha_k\}_{k \ge 0}$, $\{\beta_t\}_{t \ge 0}$, and integer sequences $\{t_k\}_{k \ge 0}$ and $\{b_k\}_{k \ge 0}$.
    
    
    {\bf For $k=0,1,\ldots$, $N$:}
    
    \vgap
    
    {\addtolength{\leftskip}{0.2in}

    {\bf For $t=0,1,\ldots, t_k-1$:}
    
    }
    \vgap
    {\color{black}
    {\addtolength{\leftskip}{0.4in}
    
    \STATE Compute partial gradient approximations of $G_{0,\mu}$ w.r.t $y$:\vspace{-.2cm}
    \begin{align}
    \tilde \nabla_y G_{0,\mu}^t  &= \left[ \frac{G(x_k, y_t+\mu v^{(1)}_t, \zeta^{(1)}_t) - G(x_k, y_t,\zeta^{(1)}_t)}{\mu}\right] v^{(1)}_t, \nonumber \\
    \tilde \nabla_y G_{0,\mu}^t &= \left[ \frac{G(x_k, z_t+\mu v^{(1)}_t, \zeta^{(1)}_t) - G(x_k, z_t,\zeta^{(1)}_t)}{\mu}\right] v^{(1)}_t, \label{sgrady_first_order_g}
    \end{align}
    where $\zeta^{(1)}_t$ and $v^{(1)}_t$ are i.i.d samples from $\zeta$ and Gaussian distribution, respectively, and\vspace{-.2cm}
    \beq
    \tilde \nabla_y F_{0,\mu}^t = \left[ \frac{F(x_k, z_t + \mu v^{(1)}_t, \xi^{(1)}_t) - F(x_k, z_t,\xi^{(1)}_t)}{\mu}\right] v^{(1)}_t
    \eeq
    where $\xi^{(1)}_t$ is another i.i.d. sample from $\xi$. Set
    \vspace{-.2cm}
    \begin{align}\notag 
    z_{t+1} &= z_t- \beta_t \tilde \nabla_y G_{0,\mu}^t, \\ \label{eq:OptJointUpdate}
    y_{t+1} &= y_t- \beta_t (\lambda^{-1} \tilde \nabla_y F_{0,\mu}^t + \tilde \nabla_y G_{0,\mu}^t).
    \end{align}
    
    }
    
    {\addtolength{\leftskip}{0.2in}
    {\bf End}
    
    \STATE  Set $\bar y_k=y_{t_k}, \bar z_k = z_{t_k}$ and compute the partial gradient approximations:\vspace{-.2cm}
    \begin{align}
    \tilde \nabla_x F_{\eta,0}^{k,z}  &= \frac{1}{s_k} \sum_{i=1}^{s_k} \tilde \nabla_x F_{\eta,0}^{k,z,i}, \label{sgradfx02}\\
    \tilde \nabla_x F_{\eta,0}^{k,z,i}  &= \left[ \frac{F(x_k+\eta u^{(2)}_{k,i}, \bar z_k, \xi_{k,i}) - F(x_k, \bar z_k,\xi_{k,i})}{\eta}\right] u^{(2)}_{k,i},\label{sgradfx2}\\
    \tilde \nabla_{x} G_{\eta,0}^{k,y} &= \frac{1}{s_k} \sum_{i=1}^{s_k} \tilde \nabla_{x} G_{\eta,0}^{k,y,i}, \quad \tilde \nabla_{x} G_{\eta,0}^{k,z} = \frac{1}{s_k} \sum_{i=1}^{s_k} \tilde \nabla_{x} G_{\eta,0}^{k,z,i}  \label{sgradgx02}, \\
    \tilde \nabla_x G_{\eta,0}^{k,y,i}  &= \left[ \frac{G(x_k + \eta u^{(2)}_{k,i}, \bar y_k, \zeta_{k,i}) - G(x_k, \bar y_k,\zeta_{k,i})}{\eta}\right] u^{(2)}_{k,i}, \label{sgradgx21} \\
    \tilde \nabla_x G_{\eta,0}^{k,z,i}  &= \left[ \frac{G(x_k + \eta u^{(2)}_{k,i}, \bar z_k, \zeta_{k,i}) - G(x_k, \bar z_k,\zeta_{k,i})}{\eta}\right] u^{(2)}_{k,i},\label{sgradgx22}
    \end{align}
    where $\zeta_{k,i}$, $\xi_{k,i}$, and $(u^{(2)}_k,v^{(2)}_{k,i})$ are i.i.d samples from $\zeta$, $\xi$, and Gaussian distribution, respectively. Set\vspace{-.2cm}
    \beqa
    \tilde \nabla \mathcal{L}_{\eta,\mu}^k \equiv  \tilde \nabla_x F_{\eta,0}^{k,z}  + \lambda (\tilde \nabla_{x} G_{\eta,0}^{k,z} - \tilde \nabla_{x} G_{\eta,0}^{k,y}), \label{hypergrad_first_orde}
    \eeqa
    and then
    \beq \label{def_xk_st2}
    x_{k+1} = \arg\min_{x \in X} \left\{\langle \tilde \nabla \mathcal{L}_{\eta,\mu}^k,x -x_k \rangle + \frac{1}{2 \alpha_k}\|x-x_k\|^2 \right\}.
    \eeq
    
    }
}

{\bf End}
	\end{algorithmic}
\end{algorithm}

To proceed with a convergence analysis of Algorithm~\ref{alg_ZBSA2}, we first need to generalize the results in Lemma~\ref{lemma:fully_my_lemma} and Lemma~\ref{lemma:yz:diffs} to the zeorht-order setting.
We start with defining counterparts to \eqnok{y_first} and \eqnok{z_first} as follows. 
\begin{align*}
    y_{\mu}^*(x) &:= \arg\min_{y} \lambda^{-1} f_{0,\mu}(x,y) + g_{0,\mu}(x,y) \\
    &= \arg\min_y \mathbb{E}[\lambda^{-1} F(x, y+\mu v, \xi) + G(x, y+\mu v, \zeta)], \\
    z_{\mu}^*(x) &:= \arg\min_{z} g_{0,\mu}(x,z) = \arg\min_z \mathbb{E}[G(x, z+\mu v, \zeta)].
\end{align*}
We also define the zeroth-order approximation of penalty formulation:
\begin{align*}
    \nabla \mathcal{L}_{\eta,\mu}^* (x) := \nabla_x f_{\eta,0}(x,y_\mu^*(x)) + \lambda \nabla_x g_{\eta,0} (x,y_\mu^*(x)) - \lambda \nabla_x g_{\eta,0}(x,z_\mu^*(x)).
\end{align*}
We can now establish the following results on the biases of the aforementioned zeroth-order approximations.
\begin{lemma}
    \label{lemma:zeroth_first_order_approx_bias}
    Let $\lambda \ge 4 L_{1,f} / \lambda_g$. For all $\bar{x} \in X$, solutions from zeroth-order approximation have biases as the following:
    \begin{align*}
        &\max\left(\|z^*(x) - z^*_{\mu}(x)\|^2, \|y^*(x) - y^*_{\mu}(x)\|^2\right) \le \frac{L_{1,g}}{\lambda_g }\mu^2 m.
    \end{align*}
    Furthermore, we have
    \begin{align*}
        \| \nabla \mathcal{L}_{\eta,\mu}^*(x) - \nabla \mathcal{L}^*(x) \| &\le \mathcal{O} \left(\eta n^{3/2} + \mu m^{3/2} + \frac{1}{\lambda} \right).
    \end{align*}
\end{lemma}
\begin{proof}
    See Section \ref{sec:3proof2} of the Supplement. 
\end{proof}
Note that the surrogate objective we plan to minimize, $\mathcal{L}^*(x)$,
 is solved using zeroth-order approximation estimators, and the above lemma guarantees that the resulting solution is also an approximate stationary point of the original problem.

\subsection{Outer Loop Analysis}
We start with the analysis of the outer-loop. Let
\begin{align*}
q^*_\mu(x) := y_\mu^*(x) - z_\mu^*(x), \quad
    q_k := \bar{y}_k - \bar{z}_k.
\end{align*}
We first need to control the biases associated with $\tilde \nabla \mathcal{L}^k_{\eta,\mu}$ in \eqref{hypergrad_first_orde}. To do so, we present the following lemma:
\begin{lemma}
    \label{lemma:inexact_first_bias}
    Define the expectation of the hyper-gradient estimator in \eqref{hypergrad_first_orde} as
    \begin{align*}
        \bar{\nabla} \mathcal{L}_{\eta,\mu}^k := \nabla_x f_{\eta,0} (x_k,  \bar{y}_k) + \lambda (\nabla_x g_{\eta,0}(x_k, \bar{y}_k) - \nabla_x g_{\eta,0}(x_k, \bar{z}_k)).
    \end{align*}
    Then, for any $k$, it holds that
    \begin{align*}
        \| \bar{\nabla} \mathcal{L}_{\eta,\mu}^k - \nabla \mathcal{L}_{\eta,\mu}^*(x_k) \| &\le (L_{1,f} + \lambda L_{2,g} \|q_\mu^*(x_k)\|) \|\bar{y}_k - y_\mu^*(x_k)\| \\
        &\quad + \lambda L_{1,g} \|q_k - q_\mu^*(x_k)\| + \lambda L_{2,g} \|q_k\|^2 + \lambda L_{2,g} \|q_\mu^* (x_k)\|^2. 
    \end{align*}
\end{lemma}
\begin{proof}
    See Section \ref{sec:3proof3} of the Supplement. 
\end{proof}
Next, we control the variance of stochastic estimators $\tilde \nabla \mathcal{L}_{\eta,\mu}^k$ in the outer-loop update:
\begin{lemma}
    \label{lemma:G_diff_var}
    For mean-squared difference of gradient estimators in \eqref{hypergrad_first_orde}: 
    \begin{align*}
        \EE_{|k} \left[ \| \tilde \nabla_x G_{\eta,0}^{k,y,i} - \tilde \nabla_x G_{\eta,0}^{k, z,i} \|^2 \right] \le \mathcal{O}(\eta^2n^3 + n \|\bar{y}_k - \bar{z}_k \|^2),
    \end{align*}
    which implies that 
    \begin{align*}
        \texttt{Var}_{|k} (\tilde{\nabla}\mathcal{L}_{\eta,\mu}^k) &\le \frac{1}{s_k} \cdot \mathcal{O}(n + \lambda^2 \eta^2 n^3 + \lambda^2 n \|\bar{y}_k - \bar{z}_k\|^2).
    \end{align*}
\end{lemma}
\begin{proof}
For simplicity, we omit the subscript $k$ and superscript $i$ here whenever the context is clear. Note that for arbitrary vectors $\|u+v+w\|^2\leq 3\left(\|u\|^2 + \|v\|^2 + \|w\|^2 \right)$, and therefore,
\begin{align}\notag 
     \EE  \| \tilde \nabla_x G_{\eta,0}^y - \tilde \nabla_x G_{\eta,0}^{z} \|^2 \leq & \frac{3}{\eta^2}\EE\left\|   \left( G(x + \eta u, \bar y, \zeta) \!-\! G(x, \bar y,\zeta) \!-\! \eta \vdot{\nabla_x G(x, \bar y, \zeta)}{u}  \right) u \right \|^2  \\ \notag & \!+ \frac{3}{\eta^2}\EE\left\|   \left( G(x + \eta u, \bar z, \zeta) \!-\! G(x, \bar z,\zeta) \!-\! \eta \vdot{\nabla_x G(x, \bar z, \zeta)}{u}  \right) u \right \|^2  \\ ~&\! + 3 \EE \left\| \left\langle \nabla_x G(x, \bar y, \zeta) - \nabla_x G(x, \bar z, \zeta), u\right\rangle u\right\|^2.  \label{3:terms}
\end{align}
We bound each term on the right side of \eqref{3:terms}. By the smoothness of $G$ we have 
\begin{align*}
    \left| G(x + \eta u, \bar y, \zeta) - G(x, \bar y,\zeta) - \eta \left\langle \nabla_x G(x, \bar y, \zeta), u \right\rangle  \right| \leq \frac{\eta^2 L_{1,G} }{2}\|u\|^2, 
\end{align*}
which implies that
\begin{align}\notag 
    \frac{3}{\eta^2} & \EE\left\|   \left( G(x + \eta u, \bar y, \zeta) - G(x, \bar y,\zeta) - \eta \vdot{\nabla_x G(x, \bar y, \zeta)}{u}  \right) u \right \|^2  \\
    &\leq \frac{3\eta^2 L_{1,G}^2 }{4}\EE\|u\|^6  \leq \mathcal{O} ( \eta^2 L_{1,G}^2 n^3 ). \label{ub:2term}
\end{align}
Clearly a similar upper bound applies to the second term on the right side of \eqref{3:terms}. To bound the third term, notice that 
\begin{align}\notag 
    &\EE \left\| \left\langle \nabla_x G(x, \bar y, \zeta) - \nabla_x G(x, \bar z, \zeta), u\right\rangle u\right\|^2 \\
    &= \EE\left[ \left| \left\langle \nabla_x G(x, \bar y, \zeta) - \nabla_x G(x, \bar z, \zeta), u\right\rangle \right|^2 \|u\|^2 \right]\\ 
    \notag &\leq \left( \EE\left[ \left| \left\langle \nabla_x G(x, \bar y, \zeta) - \nabla_x G(x, \bar z, \zeta), u\right\rangle \right|^4  \right] \right)^{\frac{1}{2}} \left(\EE \|u\|^4  \right)^{\frac{1}{2}}
    \\
    &\leq (n+4) \left( \EE\left[ \left| \left\langle \nabla_x G(x, \bar y, \zeta) - \nabla_x G(x, \bar z, \zeta), u\right\rangle \right|^4  \right] \right)^{\frac{1}{2}},\label{3rd:term}
\end{align}
where we used the Cauchy–Schwarz inequality. To simplify the notation we use $\delta_\zeta = \nabla_x G(x, \bar y, \zeta) - \nabla_x G(x, \bar z, \zeta)$. For symmetric matrices $A$ and $B$ (see \S 8.2.4 of \cite{petersen2008matrix}):
\begin{equation*}
\EE u^\top A u u^\top B u = 2\tr(AB) + \tr(A)\tr(B).
\end{equation*}
Since $\delta_\zeta$ and $u$ are independent, we have
\begin{align*}
    \EE\left[ (\delta_\zeta^\top u)^4 \right] &= \EE\left[ u^\top\delta_\zeta \delta_\zeta^\top u u^\top\delta_\zeta \delta_\zeta^\top u \right] = 3\EE_\zeta  \left\| \delta_\zeta\right\|^4 \leq 3 L_{1,G}^4 \|\bar y-\bar z\|^4,
\end{align*}
which combined with \eqref{3rd:term} gives 
\begin{equation}
     \EE \left\| \left\langle \nabla_x G(x, \bar y, \zeta) - \nabla_x G(x, \bar z, \zeta), u\right\rangle u\right\|^2 \leq \sqrt{3}(n+4) L_{1,G}^2 \|\bar y-\bar z\|^2. \label{ub:3rdterm}
\end{equation}
Using \eqref{ub:3rdterm} and  \eqref{ub:2term} in \eqref{3:terms} gives 
\begin{equation}
     \EE  \| \tilde \nabla_x G_{\eta,0}^y - \tilde \nabla_x G_{\eta,0}^{z} \|^2  \leq \mathcal{O} \left( \eta^2 L_{1,G}^2 n^3 + n L_{1,G}^2 \|\bar{y} - \bar{z}\|^2 \right).
\end{equation}

To see the second inequality, we know that
\begin{align*}
    \texttt{Var}(\tilde{\nabla}\mathcal{L}_{\eta,\mu}^k) &\le \texttt{Var}(\tilde{\nabla}_x F_{\eta,0}^{k,z}) + \lambda^2 \texttt{Var}(\tilde{\nabla}_x G_{\eta,0}^{k,z} - \tilde{\nabla}_x G_{\eta,0}^{k,y}) \\
    &\le \frac{1}{s_k} \texttt{Var} (\tilde{\nabla}_x F_\eta^{k,z,i}) + \frac{\lambda^2}{s_k} \texttt{Var} (\tilde{\nabla}_x G_{\eta,0}^{k,z,i} - \tilde{\nabla}_x G_{\eta,0}^{k,y,i}) \\
    &\le \frac{1}{s_k} \left(n\sigma_{1,F}^2 + L_{1,F}^2 \eta^2 n^3 + \lambda^2 L_{1,G}^2 \eta^2 n^3 + \lambda^2 L_{1,G}^2 n  \|\bar{y}-\bar{z}\|^2 \right).
\end{align*}
\end{proof}

Given Lemma \ref{lemma:G_diff_var}, we can bound the overall estimation bias in the outer loop by $\mathcal{O}(\epsilon)$. Simultaneously, we aim to bound the number of outer iterations less than $\mathcal{O}(\epsilon^{-1})$ given the batch size $s_k = \mathcal{O}(n/\epsilon)$. Hence, to achieve the optimal dependency on $\epsilon$ in the overall complexity, we must keep the inner loop iterations less than $\mathcal{O}(\epsilon^{-1})$, while  maintaining the outer estimation-bias less than $\mathcal{O}(\epsilon)$. To do so, we must have tight control over the quantity $\|\bar{y}_k - \bar{z}_k\|^2 = \|q_k\|^2$, which we handle in the inner-loop analysis.

\subsection{Inner Loop Analysis}
We start with upper bounds on the inner loops for $t$-steps:
\begin{lemma}  
\label{lemma:control_y_z}
    Let $\beta \le c / L_{1,g}$ and $\lambda \ge L_{1,f} / (c \lambda_g)$ for some sufficiently small $c > 0$. For every $k^{th}$ step of the outer iteration, the MSE bound at the $t^{th}$ step of the inner loop with $\beta_t = \beta$, is given by
    \begin{align*}
        \EE_{|t} \|y_{t+1} - y^*_{\mu}(x_k)\|^2 &\le (1-\beta \lambda_g) \|y_t - y^*_{\mu}(x_k) \|^2 + \beta^2\cdot {\cal O}(\mu^2 m^3 + m), \\
        \EE_{|t} \|z_{t+1} - z^*_{\mu}(x_k)\|^2 &\le (1-\beta \lambda_g) \|z_t - z^*_{\mu}(x_k) \|^2 + \beta^2\cdot {\cal O}(\mu^2 m^3 + m).
    \end{align*} 
\end{lemma}
\begin{proof}
    See Section \ref{sec:3proof4} of the Supplement. 
\end{proof}

The crucial step in our proof is to show that $q_k - q_k^*$ can be much more tightly controlled rather than controlling $y_k - y^*_k$ and $z_k - z_k^*$ separately as in Lemma \ref{lemma:control_y_z}. We can show that $q_t$ converges faster due to smaller variances by canceling noise effects in $y_t$, $z_t$ when using the common noise variables. To see this, we first show a tighter control on the difference sequences of lower-level variables:
\begin{lemma}
    \label{lemma:yz_diff_descent}
    Let $\beta \le \frac{c \lambda_g}{L_{1,g}^2 m}$ for some sufficiently small constant $c > 0$. At every iteration $k$, for any $t\ge0$ with $\beta_t = \beta$, we have
    \begin{align*}
        \EE_{|t} \|y_{t+1}-z_{t+1}\|^2 &\le (1-\lambda_g \beta) \|y_t - z_t\|^2 + \mathcal{O} \left(\frac{\beta}{\lambda^2} + \frac{\beta^2 m}{\lambda^2} + \beta^2 \mu^2 m^3\right).
    \end{align*}
\end{lemma}
\begin{proof}
    See Section \ref{sec:3proof5} of the Supplement. 
\end{proof}

Below, we show the convergence of main quantity $q_t = y_t - z_t$ toward $q_\mu^*(x_k) = y_\mu^*(x_k) - z_\mu^*(x_k)$ due to the noise canceling effect:
\begin{lemma}
    Fix $k$ and $x_k$, and for all $t \ge 0$, consider the potential function 
    \begin{align}
        \mathbb{V}^L_{t} &:= \|q_{t} - q_\mu^*(x_k) \|^2 + \frac{C_y}{\lambda^2} \|y_{t} - y_\mu^*(x_k) \|^2 \nonumber \\
        &\quad + \frac{C_z}{\lambda^2} \|z_{t} - z_\mu^* (x_k)\|^2 + \beta C_q \|q_{t}\|^2, \label{eq:lower_potential}    
    \end{align} 
    where 
    \begin{align}
        C_y &:= \mathcal{O}\left( \frac{L_{1,f}^2}{\lambda_g^2} \right), \ C_z:= \mathcal{O}\left( \frac{L_{2,g}^2 L_{f,0}^2}{\lambda_g^4} \right), \ C_q:= \mathcal{O}\left( \frac{L_{1,g}^2}{\lambda_g} \right). \label{eq:C_yzq_order}
    \end{align}
    Let $c>0$ be a sufficiently small constant such that $\beta \le \frac{c \lambda_g}{L_{1,g}^2 m}$, $L_{1,g}^2 \mu^2 m^3 \le c$, and $\beta_t = \beta$ for all $t \ge 0$. Furthermore, assume $\lambda \ge \frac{1}{c} \cdot \max\left( \frac{L_{1,f}}{\lambda_g}, \sqrt{\frac{L_{1,f}}{L_{1,g}}}, \frac{\sigma_{1,F}}{\sigma_{1,G}} \right)$. Then,
    \begin{align*}
        \mathbb{V}_{t}^L = (1-\lambda_g \beta/2)^t \mathbb{V}_0^L + \mathcal{O}(\beta/\lambda^2) + \mathcal{O}(1/\lambda^4).
    \end{align*}
\end{lemma}
We note here that while our proof follows the high-level idea presented in \cite{kwon2024complexity}, our algorithm does not rely on artificial projection steps.
\begin{proof}
    For simplicity, we denote $q_\mu^*(x_k)$ as $q^*$, $y_\mu^*(x_k)$ as $y^*$, $z_\mu^*(x_k)$ as $z^*$, and $x_k$ as $x$. Note that
    \begin{align*}
        \EE_{|t} \|q_{t+1} - q^*\|^2 &= \|q_{t+1} - q_t + q_t - q^*\|^2 \\
        &= \|q_t - q^*\|^2 + \EE_{|t} \|q_{t+1} - q_t\|^2 + 2 \EE_{|t} \vdot{q_{t+1} - q_t}{q_{t} - q^*} \\
        &= \|q_t - q^*\|^2 + \beta^2 \EE_{|t} \|\lambda^{-1} \tilde{\nabla}_y F_{0,\mu}^t + \tilde{\nabla}_y G_{0,\mu}^t - \tilde{\nabla}_y G_{0,\mu}^t  \|^2 \\
        &\quad - 2\beta \vdot{\lambda^{-1} \nabla_y f_{0,\mu}(x,y_t) + \nabla_y g_{0,\mu}(x,y_t) - \nabla_y g_{0,\mu}(x,z_t)}{q_{t} - q^*}
    \end{align*}
    The key is to observe the following: 
    \begin{align*}
        \EE_{|t} &\|q_{t+1} - q^* \|^2 \\&= \|q_{t} - q^*\|^2 + \beta^2 L_{1,G}^2 \|y_t-z_t\|^2 + \mathcal{O}(\beta^2/\lambda^2) \cdot (L_{0,f}^2 + m \sigma_{1,F}^2) \\
        &\quad - 2 \beta \vdot{\nabla_y g_{0,\mu}(x,y_t) - \nabla_y g_{0,\mu}(x, z_t+q^*)}{y_t-(z_t+q^*)} \\
        &\quad - 2 \beta \vdot{\nabla_y g_{0,\mu}(x,z_t+q^*) - \nabla_y g_{0,\mu}(x,z_t) + \lambda^{-1} \nabla_y f_{0,\mu}(x,y_t)}{y_t-(z_t+q^*)}.
    \end{align*}
    From strong-monotonicity of $g_{0,\mu}$ (inherited from $g$), we have
    \begin{align*}
        \vdot{\nabla_y g_{0,\mu}(x,y_t) - \nabla_y g_{0,\mu}(x,z_t+q^*)}{y_t-(z_t+q^*)} &\ge \lambda_g \|y_t-z_t - q^*\|^2 \\&= \lambda_g \|q_t - q^*\|^2.
    \end{align*}
    Then observe that 
    \begin{align*}
        &\nabla_y g_{0,\mu}(x,z_t+q^*) - \nabla_y g_{0,\mu}(x,z_t) + \lambda^{-1} \nabla_y f_{0,\mu}(x,y_t) \\
        &\quad = \nabla^2_{yy} g_{0,\mu}(x,z_t) q^* + \mathcal{O}( L_{g,2} \|q^*\|^2)  + \lambda^{-1} \nabla_y f_{0,\mu}(x,y_t) \\
        &\quad = \underbrace{\nabla^2_{yy} g_{0,\mu}(x,z^*) q^* + \lambda^{-1} \nabla_y f_{0,\mu}(x,y^*)}_{(a): \mathcal{O}(1/\lambda^2)} + \underbrace{(\nabla^2_{yy} g_{0,\mu}(x,z_t) - \nabla^2_{yy} g_{0,\mu}(x,z^*)) q^* }_{\mathcal{O}(L_{2,g} \|z_t-z^*\|/ \lambda) } \\
        &\quad \quad + \mathcal{O}(L_{2,g} \underbrace{\|q^*\|^2}_{(b): \mathcal{O}(1/\lambda^2)}) + \lambda^{-1} \underbrace{(\nabla_y f_{0,\mu}(x,y_t) - \nabla_y f_{0,\mu}(x,y^*))}_{\mathcal{O}(L_{1,f} \|y_t-y^*\|)},
    \end{align*}
    where (a) and (b) come from Lemma \ref{lemma:yz_diff_mu} of the Supplement, which is a generalization of Lemma \ref{lemma:yz:diffs}. Thus, we have shown that
    \begin{align}
        \EE_{|t}\|q_{t+1} - q^*\|^2 &= (1-\lambda_g \beta) \|q_t - q^*\|^2 + \beta^2 L_{1,G}^2 \|y_t - z_t\|^2 + \mathcal{O}\left(\frac{\beta L_{0,f}^2 + \beta^2 m \sigma_{1,F}^2}{\lambda^2}\right)  \nonumber\\
        &\quad + \mathcal{O} \left( \beta \frac{L_{1,g}^2 L_{2,g}^2 L_{0,f}^4}{\lambda^4 \lambda_g^7} \right) + \mathcal{O}\left(\beta \frac{L_{2,g}^2 L_{0,f}^2}{\lambda^2 \lambda_g^3} \right) \|z_t - z^*\|^2 \nonumber \\
        &\quad + \mathcal{O}\left(\beta \frac{L_{1,f}^2}{\lambda^2 \lambda_g} \right) \|y_t - y^*\|^2. \label{eq:q_descent_inequality}
    \end{align}
    Recalling the results from Lemma \ref{lemma:control_y_z} and \ref{lemma:yz_diff_descent}, we have (omitting higher-order terms, see \ref{sec:3proof5} for detailed constants) 
    \begin{align*}
        \|q_{t+1}\|^2 - \|q_t\|^2 &\le -\beta\lambda_g  \|q_t\|^2 + \mathcal{O}\left( \frac{ \beta m L_{0,f}^2}{ \lambda^2 \lambda_g} \right), \\
        \|y_{t+1} - y^*\|^2 - \|y_t - y^*\|^2 &\le -\beta\lambda_g \|y_t - y^*\|^2 + \mathcal{O}(\beta^2 m \sigma_{1,G}^2), \\
        \|z_{t+1} - z^*\|^2 - \|z_t - z^*\|^2 &\le -\beta\lambda_g \|z_t - z^*\|^2 + \mathcal{O}( \beta^2 m \sigma_{1,G}^2).
    \end{align*}    
    Combining the above inequalities with \eqref{eq:q_descent_inequality}, we get 
    \begin{align*}
        \mathbb{V}_{t+1}^L - \mathbb{V}_t^L &\le
        (\|q_{t+1} - q_\mu^*(x_k)\|^2 - \|q_{t} - q_\mu^*(x_k)\|^2) + \beta C_q (\|q_{t+1}\|^2 - \|q_t\|^2) \\
        &\ + \frac{C_y}{\lambda^2} (\|y_{t+1} - y_{\mu}^*(x_k)\|^2 - \|y_{t} - y_{\mu}^*(x_k)\|^2) \\
        &\ + \frac{C_y}{\lambda^2} (\|z_{t+1} - z_{\mu}^*(x_k)\|^2 - \|z_{t} - z_{\mu}^*(x_k)\|^2) \\
        &\le -\lambda_g \beta \|q_t - q_\mu^*(x_k)\|^2 - \left(\frac{C_y}{\lambda^2} \lambda_g \beta - \mathcal{O}\left( \frac{\beta L_{1,f}^2}{\lambda^2 \lambda_g} \right) \right) \|y_t - y^*_\mu(x_k)\|^2 \\
        &\ - \left(\frac{C_z}{\lambda^2} \lambda_g \beta - \mathcal{O} \left(\frac{\beta L_{0,f}^2 L_{2,g}^2 }{\lambda^2 \lambda_g^3} \right) \right) \|z_t - z^*_\mu(x_k)\|^2 \\
        &\ - \beta^2 (C_q \lambda_g - L_{1,G}^2 ) \|q_t\|^2 + \mathcal{O} \left( \beta \frac{L_{1,g}^2 L_{2,g}^2 L_{0,f}^4 }{\lambda^4 \lambda_g^7} \right) \\
        &\ + \beta^2 \cdot \mathcal{O} \left(\frac{m L_{0,f}^2 C_q}{\lambda^2 \lambda_g } + \frac{m \sigma_{1,G}^2 (C_y + C_z)}{\lambda^2}  \right) \\
        &\le -\frac{\lambda_g \beta}{2} \mathbb{V}_t^L + \mathcal{O} \left( \beta \frac{L_{1,g}^2 L_{2,g}^2 L_{0,f}^4 }{\lambda^4 \lambda_g^7} \right) \\ 
        & \ + \mathcal{O} \left(\frac{\beta^2 m}{\lambda^2}  \left( \frac{L_{0,f}^2 L_{1,G}^2}{\lambda_g^2} + \frac{\sigma_{1,G}^2 L_{1,f}^2}{\lambda_g^2} + \frac{\sigma_{1,G}^2 L_{0,f}^2 L_{2,g}^2}{\lambda_g^4} \right) \right),
    \end{align*}
    with sufficiently large constants $C_y, C_z, C_q$ satisfying \eqref{eq:C_yzq_order}. 
    Unfolding this inequality recursively, we obtain (omitting constants in higher-order terms)
    \begin{align*}
        \mathbb{V}_{t+1}^L \le (1-\beta\lambda_g/2)^t \mathbb{V}_0^L +  \frac{1}{\lambda^4} \cdot \mathcal{O}\left(\frac{L_{1,g}^2 L_{2,g}^2 L_{f,0}^4}{\lambda_g^8} \right) + \mathcal{O}\left(\frac{\beta m}{\lambda^2} \right),
    \end{align*}
    which yields $\|q_{t_k} - q_k^*\|^2 = \mathcal{O}(\beta/\lambda^2) + \mathcal{O}(1/\lambda^4)$ with $t_k = \mathcal{O}(1/\beta)$, justiying the choice $\lambda^2 = \mathcal{O} (\epsilon^{-1}), t_k = \beta^{-1} = \mathcal{O} (m \epsilon^{-1})$.
\end{proof}

\subsection{Convergence Analysis}
To do so, we require boundedness of  $4^{th}$-order moment of stochastic gradients of $f$ from Assumption~\ref{assumption:fourth_moment} which implies that,
for all $x \in X, y \in \mathbb{R}^{m}$, 
    \begin{align*}
        \EE \| \nabla_y F (x,y; \zeta) - \nabla_y f(x,y) \|^4 \le \mathcal{O} ( \sigma_{1,F}^4).
    \end{align*}
With this assumption, we can now bound the descent lemma for fourth-moment of $y$ and $z$ differences:
\begin{lemma}\label{descentLemmaFourthMoment}
    The fourth-moment of gradient estimators in \eqref{hypergrad_first_orde} is bounded as the following:
    \begin{align*}
        \EE_{|t} \|\tilde{\nabla}_y G_{0,\mu}^t - \tilde{\nabla}_z G_{0,\mu}^t\|^4 &\le \mathcal{O}(\eta^4 m^6) + \mathcal{O}(m^2) \|y_t - z_t\|^4. 
    \end{align*}
    Furthermore, suppose Assumption \ref{assumption:fourth_moment} holds. Let $\beta \le c \cdot \left( \lambda_g / (m L_{1,G}^2 + m \sigma_{1,F}^4) \right)$ and $\mu^2 \le \frac{c}{m^{3} L_{1,G}}$ for sufficiently small $c>0$. Then, the following holds:
    \begin{align*}
        \EE_{|t} \|y_{t+1} - z_{t+1}\|^4 &\le (1-\lambda_g \beta) \|y_t - z_t\|^4 + \mathcal{O}(\beta / \lambda^4). 
    \end{align*}    
\end{lemma}
\begin{proof}
    See Section \ref{sec:3proof6} of the Supplement. 
\end{proof}

Now, combining all results, we can state our main theorem on the convergence of penalty-based zeroth-order methods: 
\begin{theorem}(Convergence of Algorithm \ref{alg_ZBSA2})
    Suppose that $\{\bar{y}_k, \bar{z}_k, x_k\}_{k\ge 0}$ is generated by Algorithm \ref{alg_ZBSA2}, Assumptions \ref{fg_assumption}-\ref{assumption:fourth_moment} hold, and 
    \begin{align*}
        \alpha = \frac{1}{5 L_{1,\psi}}, \ \beta = \mathcal{O} \left( \frac{\epsilon}{m} \right), \\
        \lambda^2 = \Omega(\epsilon^{-1}), \ s = \mathcal{O}(n / \epsilon), \ t = \mathcal{O}(m / \epsilon), \\
        \eta = \mathcal{O} \left(\sqrt{\min \left(\frac{1}{\lambda^2 n^3}, \frac{\epsilon}{n^3} \right) } \right), \ \mu= \mathcal{O} \left(\sqrt{\epsilon / m^3}\right),
    \end{align*}
    with a sufficiently small target accuracy $\epsilon > 0$. Let $\alpha_k = \alpha, \beta_k = \beta, s_k = s, t_k = t$ for all $k \ge 0$. Then after $N$ outer iterations, the following holds:
    \begin{align*}
        \EE\|\nabla \psi(x_R)\|^2 &\le \frac{\mathbb{V}^U_0}{\alpha N} + \mathcal{O} \left(\beta m + \frac{\alpha n}{s} + \frac{1}{\lambda^2} \right),
    \end{align*}
    where $R$ is a random variable following the distribution in \eqref{def_R}, and $\mathbb{V}^U_{k}$ is the outer potential function defined as
    \begin{align*}
        \mathbb{V}^U_{k} &:= \psi(x_k) - \psi(x^*) + \alpha \lambda^2 \cdot \mathcal{O} \left( \mathbb{V}^L_k + \frac{\alpha n}{s} \|q_k\|^2 + \|q_k\|^4\right),
    \end{align*}
    where $q_k := \bar{y}_k - \bar{z}_k$ and $\mathbb{V}_k^L$ is defined in \eqref{eq:lower_potential} at $t=0$ for each $k$.
\end{theorem}
\begin{proof}
We first note that combining Lemma \ref{lemma:fully_my_lemma}, \ref{lemma:zeroth_first_order_approx_bias}, and \ref{lemma:inexact_first_bias}, the bias in hyper-gradient estimation can be bounded by (omitting smoothness and strong-convexity parameters):
\begin{align*}
    \|\nabla \psi(x_k) - \bar{\nabla} \mathcal{L}_{\eta,\mu}^k\| &\le \|\nabla \psi(x_k) - \nabla \mathcal{L}^*(x_k) \| \\
    &\quad + \| \nabla \mathcal{L}^*(x_k) - \nabla \mathcal{L}_{\eta, \mu}^*(x_k) \| + \| \nabla \mathcal{L}_{\eta, \mu}^*(x_k) - \bar{\nabla} \mathcal{L}_{\eta,\mu}^k\|
    \\ &\le \mathcal{O} (\eta n^{3/2} + \mu m^{3/2}) + \mathcal{O} (1) \|\bar{y}_k - y_{\mu}^*(x_k)\| \\
    &\quad + \lambda L_{1,g} \|q_k-q_\mu^*(x_k)\| + \lambda L_{2,g} \|q_k\|^2 + \mathcal{O} (1 / \lambda),
\end{align*}
and the variance:
\begin{align*}
    \texttt{Var}(\tilde{\nabla} \mathcal{L}_{\eta,\mu}^k) &\le \frac{\mathcal{O}(1)}{s_k} (n \sigma_{1,F}^2 + \lambda^2 \eta^2 n^3 + \lambda^2 n \|q_k\|^2).
\end{align*}
We start with the smoothness of $\psi(x)$:
\begin{align*}
    \psi(x_{k+1}) - \psi(x_k) &\le \vdot{\nabla \psi_k}{ x_{k+1} - x_k} + \frac{L_{1,\psi}}{2} \|x_{k+1} - x_k\|^2 \\
    &= -\alpha \vdot{\nabla \psi_k}{ \tilde \nabla \mathcal{L}_{\eta,\mu}^k } + \frac{L_{1,\psi} \alpha^2}{2} \|\tilde \nabla \mathcal{L}_{\eta,\mu}^k \|^2.
\end{align*}
Taking expectation on both sides 
\begin{align*}
    &\EE_{|k} [\psi(x_{k+1})] - \psi(x_k) \\
    &\le -\alpha \|\nabla \psi_k\|^2 + \alpha \vdot{\nabla \psi_k}{ - \bar{\nabla} \mathcal{L}_{\eta,\mu}^k + \nabla \psi_k} + \frac{L_{1,\psi} \alpha^2}{2} (\|\bar \nabla \mathcal{L}_{\eta,\mu}^k \|^2 + \texttt{Var}(\tilde{\nabla} \mathcal{L}_{\eta,\mu}^k) ) \\
    &\le -\frac{\alpha}{2} \|\nabla \psi_k\|^2 + \mathcal{O}(1) \cdot \alpha \|\nabla \psi_k - \bar \nabla \mathcal{L}_{\eta, \mu}^k\|^2 + \frac{L_{1,\psi} \alpha^2}{2} \texttt{Var}(\tilde{\nabla} \mathcal{L}_{\eta,\mu}^k) \\
    &\le -\frac{\alpha}{2} \|\nabla \psi_k\|^2 + \mathcal{O}(\alpha) \left(\eta^2 n^3 + \mu^2 m^3 + \|\bar{y}_k - y_{\mu}^*(x_k)\|^2 \right) \\
    &\quad + \mathcal{O}(\alpha) \left(\lambda^2 L_{1,g}^2 \|q_k - q_\mu^*(x_k)\|^2 + \lambda^2 L_{2,g}^2 \|q_k\|^4 + \frac{1}{\lambda^2}   \right) \\
    &\quad + \frac{\mathcal{O}(\alpha^2)}{s} \left(n\sigma_{1,F}^2 + \lambda^2 \eta^2 n^3 + \lambda^2 n \|q_k\|^2 \right),
\end{align*}
given that $\alpha \ll 1/L_{1,\psi}$ and $s = \mathcal{O}(n/\epsilon)$. 
The list of inequalities we have is the following (omitting smoothness and strong-convexity parameters):
\begin{align*}
    \|q_{k+1}\|^2 &\le (1-\lambda_g \beta)^{t} \|q_k\|^2 + \mathcal{O}(1/\lambda^2), \\
    \|q_{k+1}\|^4 &\le (1-\lambda_g \beta)^{t} \|q_k\|^4 + \mathcal{O} (1/\lambda^4), \\
    \|q_{k+1}-q^*_\mu(x_{k+1})\|^2 \le \mathbb{V}_{k+1}^L &\le (1-\lambda_g \beta / 4)^{t} \mathbb{V}_k^L + \mathcal{O} (\beta m/\lambda^2) + \mathcal{O} (1/\lambda^4), \\
    \|\bar{y}_{k+1} - y_{\mu}^*(x_{k+1})\|^2 &\le (1-\lambda_g \beta)^{t} \|\bar{y}_{k} - y_{\mu}^*(x_k)\|^2 +  \mathcal{O}(\beta m),
\end{align*}
where $\mathbb{V}_k^L$ is defined in \eqref{eq:lower_potential} at $t=0$.

Now, let $t_k = \mathcal{O}(\lambda_g^{-1} \beta^{-1}) = \mathcal{O}(m/\epsilon)$. Putting all inequalities together it follows that 
\begin{align*}
    \mathbb{V}_{k+1}^U - \mathbb{V}_k^U \le -\alpha \|\nabla \psi_k \|^2 + \mathcal{O} \left( \alpha \beta m + \frac{\alpha^2 n}{s} + \frac{\alpha}{\lambda^2} \right).
\end{align*}
Telescoping sum from $k=0...N-1$ gives
\begin{align*}
    \sum_{k=0}^{N-1} \alpha \cdot \|\nabla \psi_k\|^2 \le \mathbb{V}_0^U + N\cdot \mathcal{O} \left(\alpha\beta m + \frac{\alpha^2 n}{s} + \frac{\alpha}{\lambda^2} \right).
\end{align*}
Dividing both sides by $\alpha N$, verifies the theorem's claim.
\end{proof}

From the theorem, we can conclude that to obtain an $\epsilon$-stationary point of the BLP by Algorithm~\ref{alg_ZBSA2}, the total outer-loop iterations should be $N = \mathcal{O}(\epsilon^{-1})$ with $\alpha = \mathcal{O}(1)$, yielding $Ns = \mathcal{O}(n/\epsilon^2)$ oracles access in the outer-loop iterations, and $Nt = \mathcal{O}(m/\epsilon^2)$ in the inner-loop overall. Thus, the total sample complexity of the algorithm is bounded by $\mathcal{O}((n+m)/\epsilon^2)$, which is optimal in terms of dependence on both problem dimension and target accuracy.

\section{Numerical Experiments}\label{numeric}
In this section, we present an experimental study aimed at evaluating the practical effectiveness of the proposed algorithms. Our focus is on the hyper-representation problem, which is commonly examined in prior bilevel programming studies \cite{franceschi2018bilevel,sow2022convergence,AghaGhad25}. In this setting, a regression (or classification) model is learned through a bilevel formulation, where the lower-level problem optimizes the parameters of a linear regressor, and the upper-level problem determines the embedding model parameters (i.e., the representation). The underlying BLP can be expressed as:
\begin{align*} \nonumber
&\minimize_{x} \left\{\psi(x):= f(x,y^*)= \frac{1}{2n_1}\left\|T(\chi_1;x)y^* - b_1 \right\|^2+\frac{\gamma}{2}\|x\|^2\right\}\\ 
& \text{subject to:} \ \  y^* = \argmin_y \ \frac{1}{2n_2} \left\|T(\chi_2;x)y - b_2 \right\|^2 + \frac{\gamma}{2}\|y\|^2.
\end{align*}
Here, $T(\chi_i, x)$ denotes a two-layer network, where $\chi_1$ and $\chi_2$ correspond to the training and validation datasets, respectively, and $b_1 \in \mathbb{R}^{n_1}$ and $b_2 \in \mathbb{R}^{n_2}$ are their associated response vectors. All experiments are conducted with $x \in \mathbb{R}^{64 \times 128}$, $y \in \mathbb{R}^{128}$, $n_1 = n_2 = 500$, $\gamma = 10^{-6}$, and smoothing parameters $\eta_1 = \eta_2 = \mu_1 = \mu_2 = 10^{-5}$. Rather than direct queries to $f$ and $g$, their stochastic oracles $F$ and $G$ are computed using the loss over a uniformly random subset of 5 rows from $T(\chi_i, x)$ and $b_i$ (corresponding to 1\% of the total sample size). Convergence plots report the norm of the hyper-gradient $\nabla \psi(x)$ against the number of oracle queries to $F$ and $G$. Each experiment is repeated 10 times, with the solid curve denoting the average performance and the shaded region representing the range across trials. The initializations of $x$ and $y$ are randomized with a fixed seed, ensuring identical starting points for all experiments. Finally, to provide a consistent measure of oracle complexity across different algorithms, query counts are scaled by batch size: for example, evaluating the objective on a batch of size 100 is recorded as 100 query calls, even though practically a vectorized implementation requires only one objective evaluation. 

We begin by comparing Algorithm \ref{alg_ZBSA} with the ZDSBAF in \cite{AghaGhad25}, which, to the best of our knowledge, is the first fully zeroth-order framework proposed for bilevel programs. In contrast to ZDSBAF, Algorithm \ref{alg_ZBSA} employs a mini-batch strategy to accelerate convergence ($s_k\gg 1$), and executes longer inner-loop iterations (larger $t_k$). As supported by the theory, these two strategies are sufficient to make a major improvement in convergence. To test this numerically, in an experiment, Algorithm \ref{alg_ZBSA} uses $t_k=512$, and a mini-batch of size $s_k=100$, while the ZDSBAF algorithm uses a shallower inner loop update ($t_k=64$) and no mini-batch sampling. Figure \ref{fig1} shows that with significantly fewer queries to $F$ and $G$, Algorithm \ref{alg_ZBSA} attains a much smaller hyper-gradient. 
\begin{figure}[!htbp]
\begin{center}
\begin{overpic}[trim={0.15cm -.25cm  0 0},clip,height=2.2in]{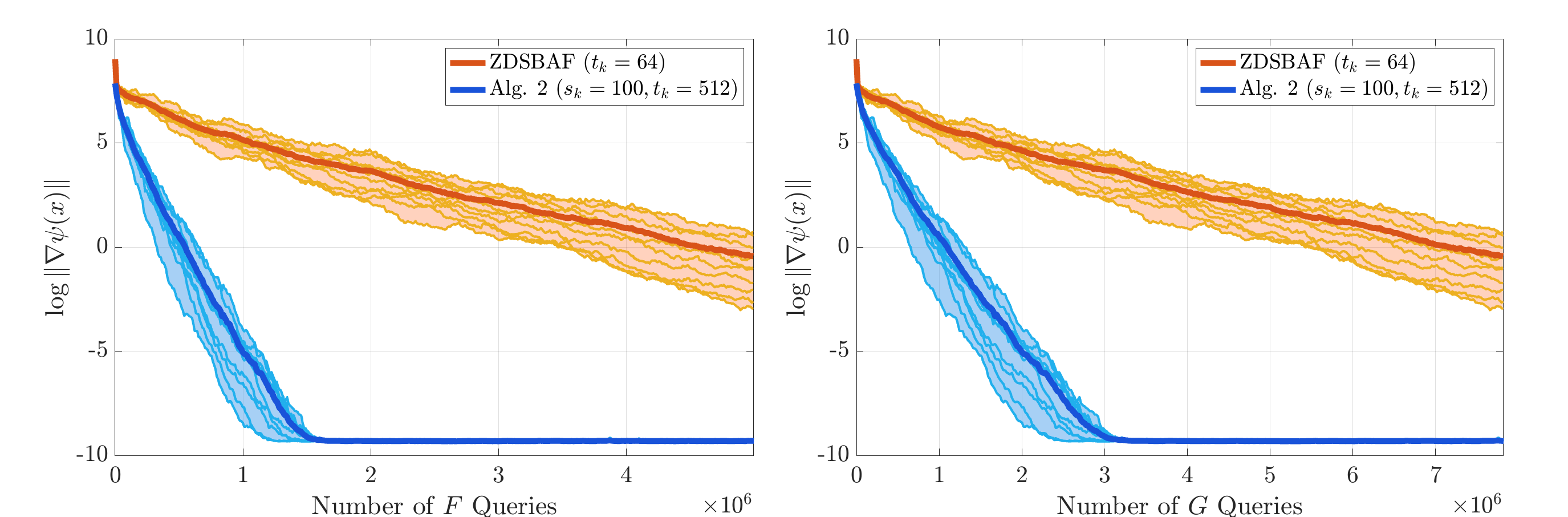}
\end{overpic}
\end{center}

		\caption{Comparing Algorithm \ref{alg_ZBSA} with the ZDSBAF framework in \cite{AghaGhad25}}\label{fig1}
\end{figure}

Next, we numerically analyze the affect of $s_k$ and $t_k$ on the convergence of Algorithm \ref{alg_ZBSA}. Interestingly, letting the inner-loop mature more by picking larger $t_k$ tends to improve the convergence more noticeable than the batch size. This has been demonstrated in Figure \ref{fig2} by varying $t_k$ from 64 to 256 and observing the increasing contrast with the ZDSBAF performance. On the other hand, beyond a small threshold (approximately 10), the batch size $s_k$ does not significantly improve the convergence speed, but rather reduces the variance in the convergence pattern as shown in Figure \ref{fig3}.

\begin{figure}[!htbp]
\begin{center}
\begin{overpic}[trim={0 -1.45cm  0 0},clip,height=2.2in]{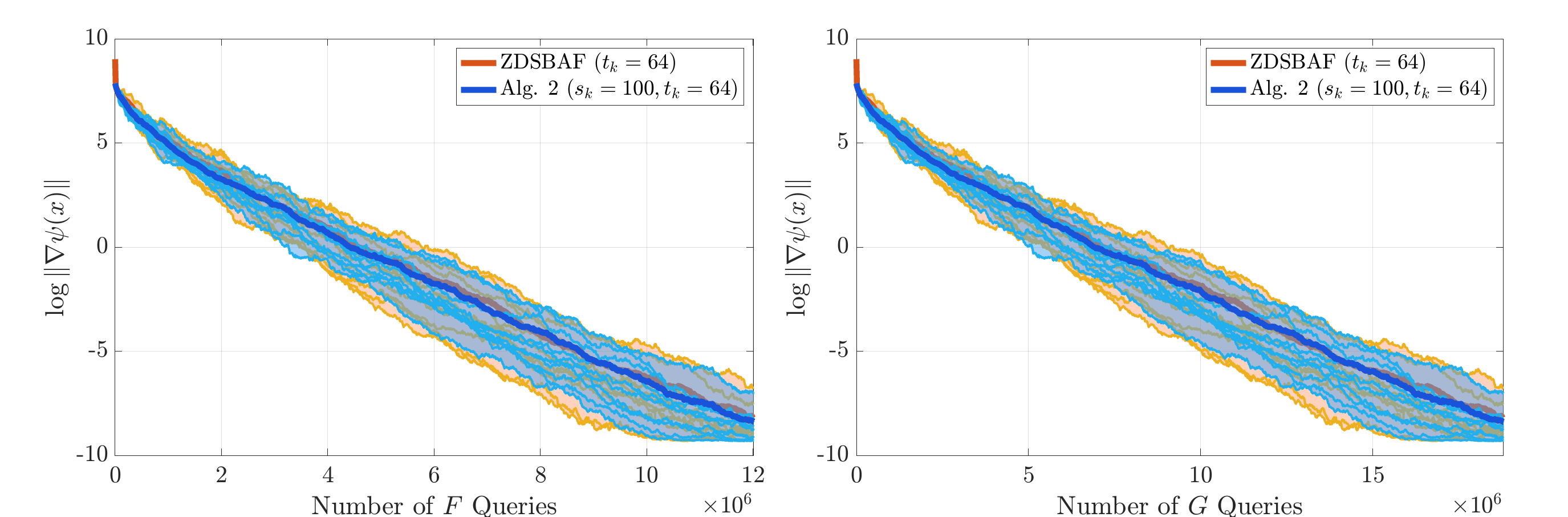}
\put (50,1) {\scalebox{.8}{\rotatebox{0}{(a)}}}
\end{overpic}\\
\begin{overpic}[trim={0 -1.45cm  0 0},clip,height=2.2in]{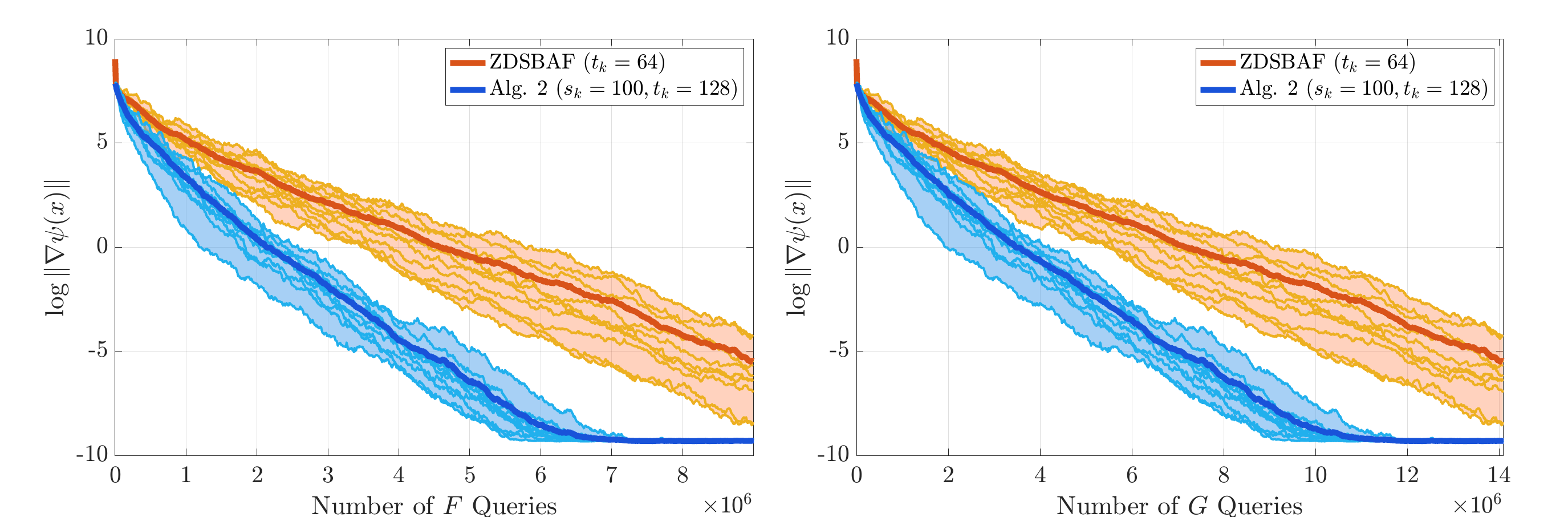}
\put (50,1) {\scalebox{.8}{\rotatebox{0}{(b)}}}
\end{overpic}\\
\begin{overpic}[trim={0 -1.45cm  0 0},clip,height=2.2in]{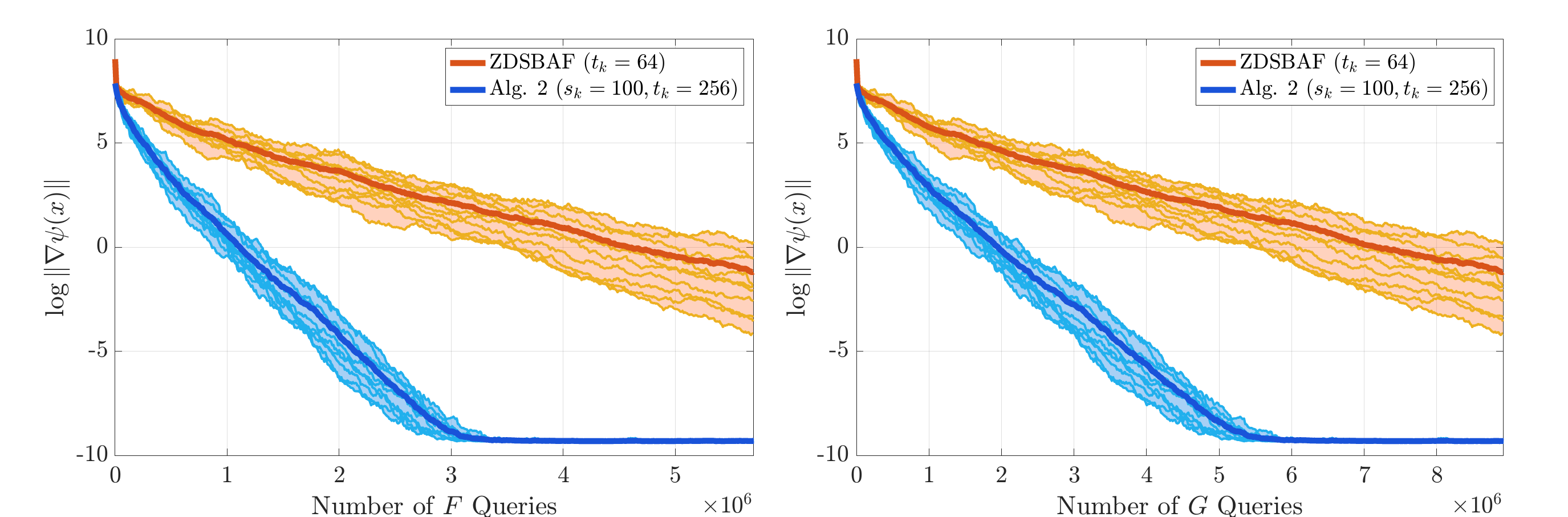}
\put (50,1) {\scalebox{.8}{\rotatebox{0}{(c)}}}
\end{overpic}
\end{center}

		\caption{ Comparing Algorithm \ref{alg_ZBSA} for different $t_k$ with the ZDSBAF framework in \cite{AghaGhad25} }\label{fig2}
\end{figure}

\begin{figure}[!htbp]
\begin{center}
\begin{overpic}[trim={0.15cm -.25cm  0 0},clip,height=2.2in]{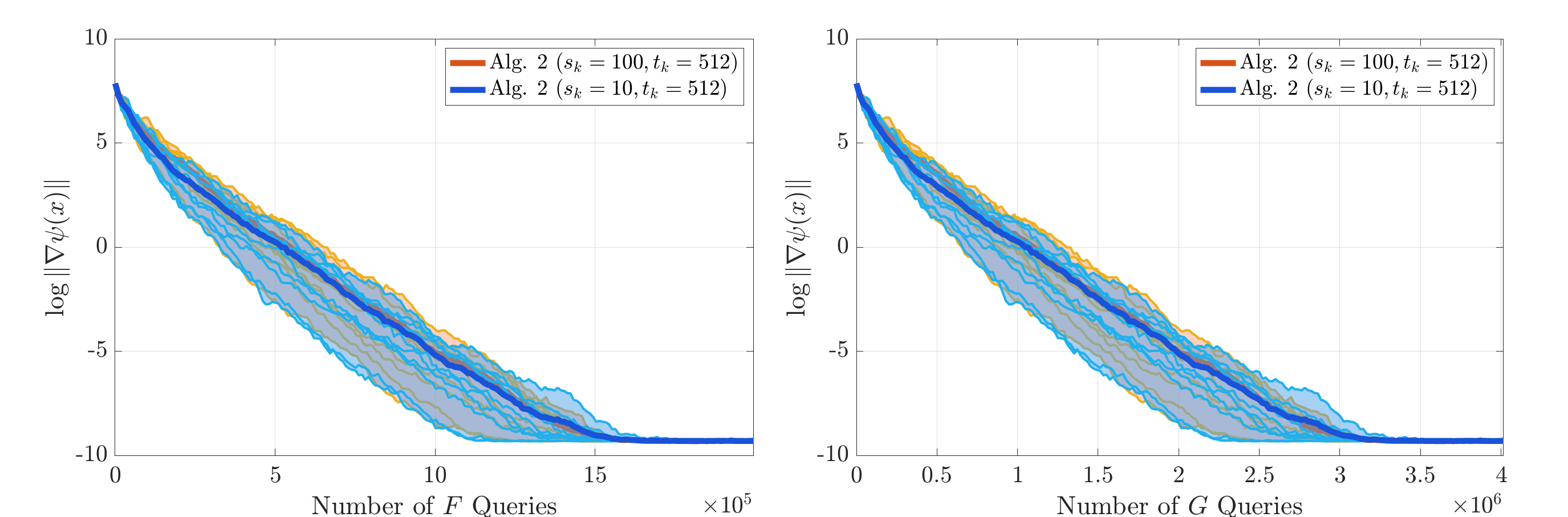}
\end{overpic}
\end{center}

		\caption{The performance of Algorithm \ref{alg_ZBSA} for different batch sizes}\label{fig3}
\end{figure}

Finally, we compare the performance of Algorithm \ref{alg_ZBSA} with Algorithm \ref{alg_ZBSA2}, whose optimality has been theoretically established. Figure \ref{fig4} shows the convergence of the two algorithms when the number of inner-loop iterations is fixed at $t_k=512$. Unlike Algorithm \ref{alg_ZBSA}, the effectiveness of Algorithm \ref{alg_ZBSA2} is less sensitive to the maturity of the inner-loop iterations. As the figure further illustrates, even when $t_k=64$, the performance of Algorithm \ref{alg_ZBSA2} remains close to that of $t_k=512$.
\begin{figure}[!htbp]
\begin{center}
\begin{overpic}[trim={0.15cm -.25cm  0 0},clip,height=2.2in]{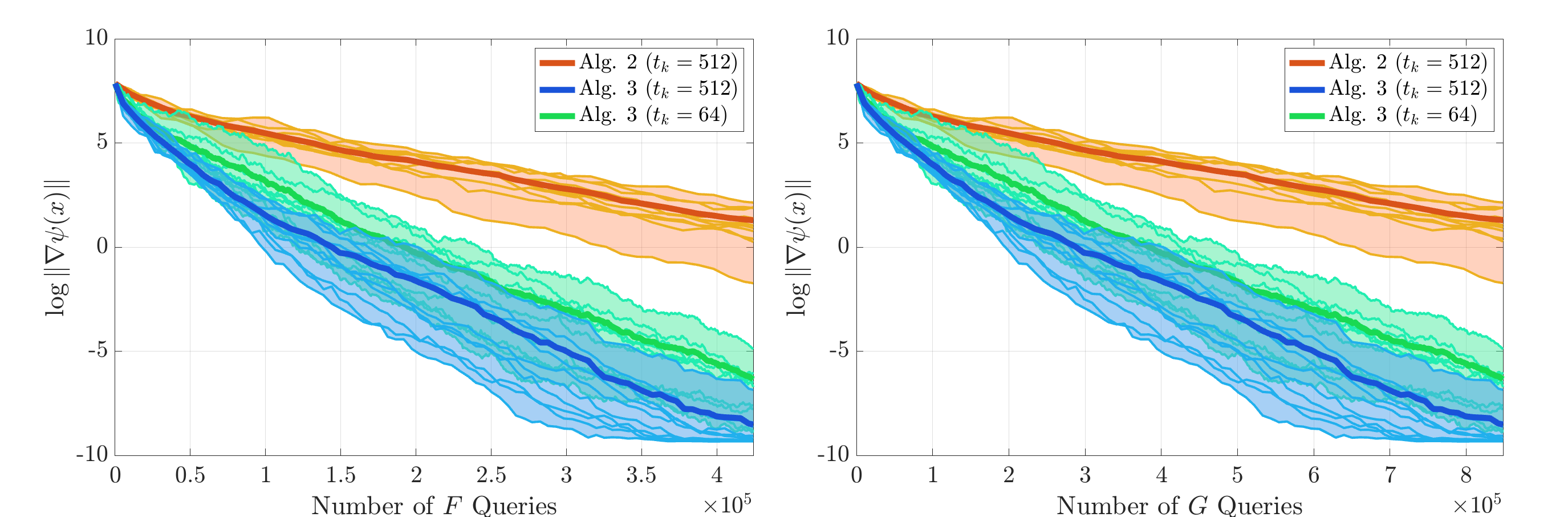}
\end{overpic}
\end{center}

		\caption{Comparing Algorithm \ref{alg_ZBSA} with Algorithm \ref{alg_ZBSA2} for identical and different $t_k$}\label{fig4}
\end{figure}

In all previous experiments with Algorithm \ref{alg_ZBSA}, we set $T=1000$ for the number of iterations used to approximate the Hessian inverse in Algorithm \ref{Hinv:alg1}, consistent with the theoretical requirement that $T$ scales cubically with the dimension. In practice, however, much smaller values of $T$ can be sufficient. Figure \ref{fig5} illustrates that as $T$ decreases, the performance of Algorithm \ref{alg_ZBSA} surpasses that of Algorithm \ref{alg_ZBSA2}. In particular, for smaller $T$, Algorithm \ref{alg_ZBSA} requires substantially fewer queries to $F$. This advantage stems from equation \eqref{eq:OptJointUpdate}, where the $y$-update in Algorithm \ref{alg_ZBSA2} requires queries to both functions, whereas in Algorithm \ref{alg_ZBSA} the $y$-update only involves queries to $G$, as shown in \eqref{def_yt}. This is an indication that while Algorithm \ref{alg_ZBSA2} is theoretically optimal, it is possible that for some problems where the Hessian inverse can be treated coarsely, we observe a faster convergence of Algorithm \ref{alg_ZBSA}, especially when the quesries to $F$ are more expensive.

\begin{figure}[!htbp]
\begin{center}
\begin{overpic}[trim={0.15cm -.25cm  0 0},clip,height=1.85in]{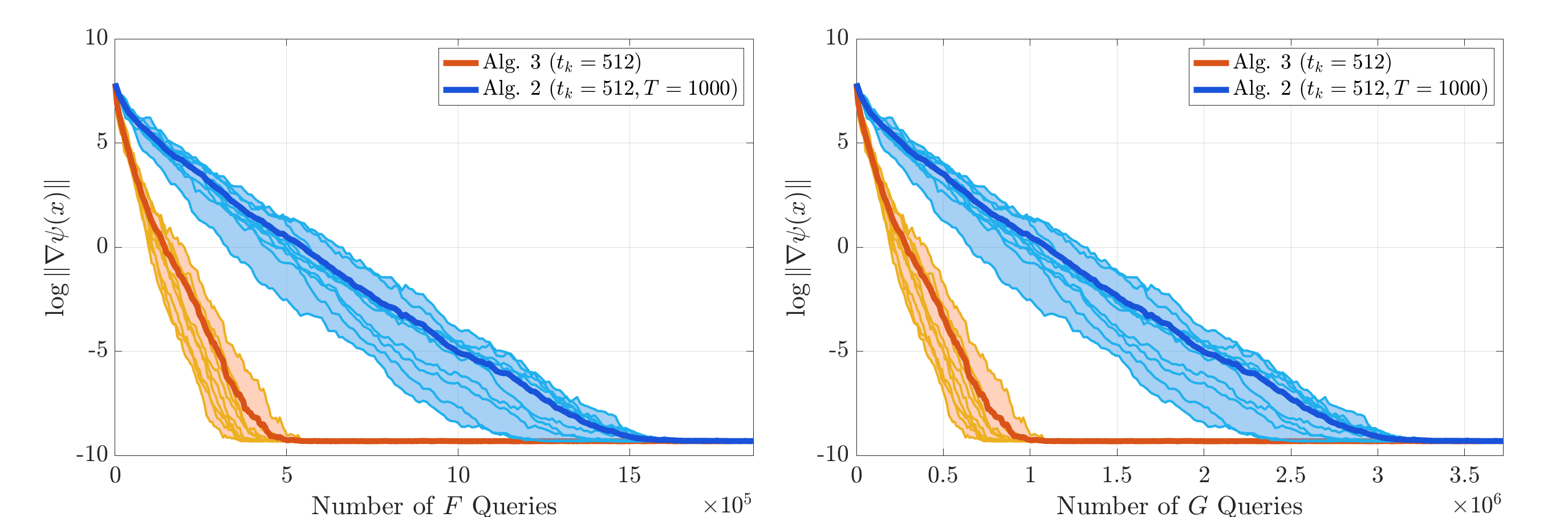}

\end{overpic}\hspace{.0cm}
\begin{overpic}[trim={0 -.25cm  0 0},clip,height=1.85in]{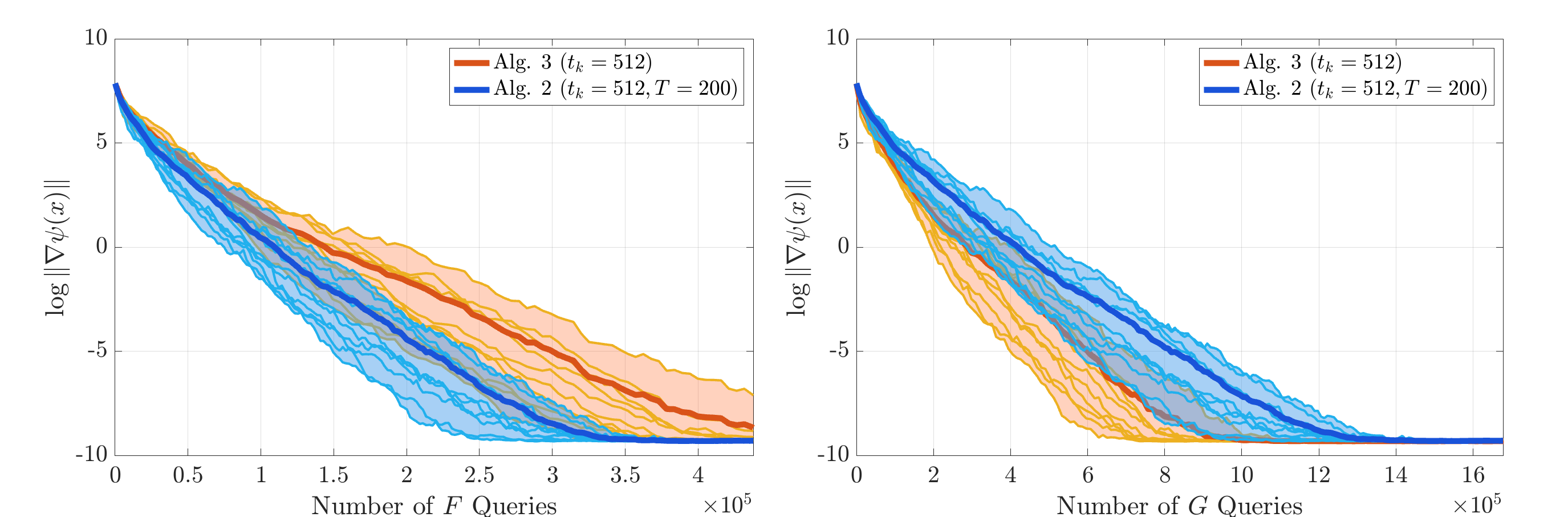}
\end{overpic}

\begin{overpic}[trim={0 -.25cm  0 0},clip,height=1.85in]{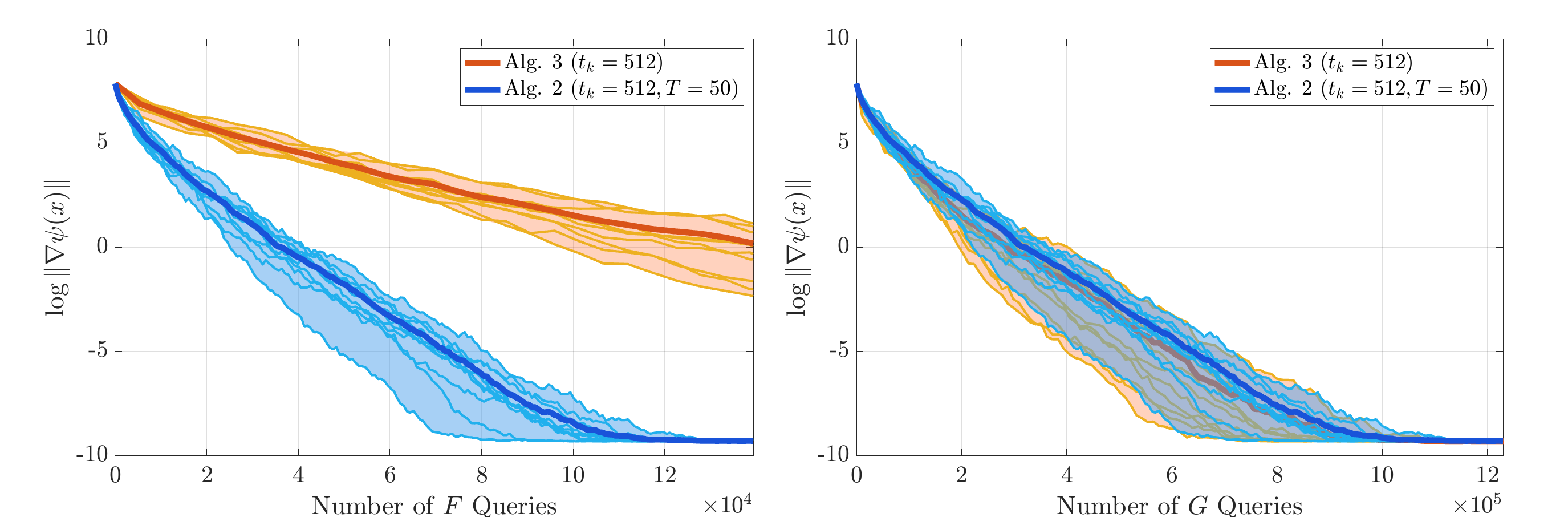}
\end{overpic}\hspace{.0cm}

\begin{overpic}[trim={0 -.25cm  0 0},clip,height=1.85in]{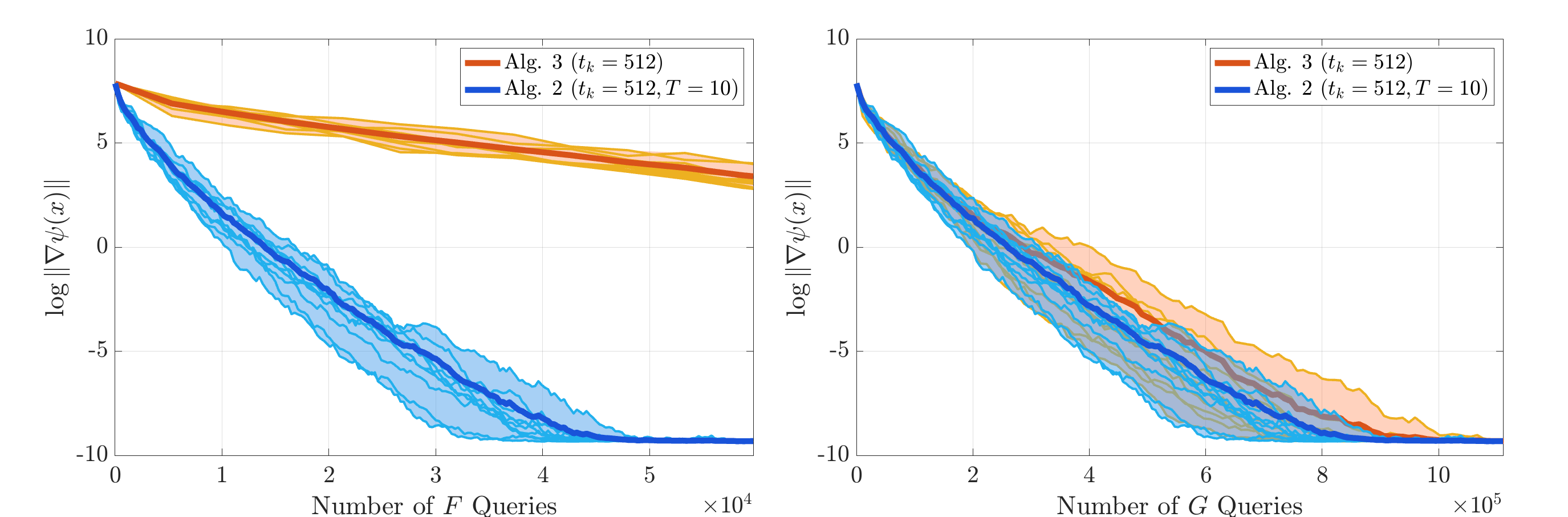}
\end{overpic}\hspace{.0cm}

\end{center}

		\caption{Comparing Algorithm \ref{alg_ZBSA} with Algorithm \ref{alg_ZBSA2} as $T$ decreases}\label{fig5}
\end{figure}

\section{Conclusion}
In this paper, we provide (nearly) optimal zeroth-order stochastic approximation algorithms for solving stochastic bilevel optimization problems. in particular, our first proposed algorithm uses a double-loop framework and mini-batch of samples from queries to both upper and lower level objective functions and can achieve an optimal sample complexity (up to a logarithmic factor) for finding a stationary point of the problem in terms of the dependence on the target accuracy while cubically depends on the problem dimension. Since this dependency comes from approximating second-order derivatives by zeroth-order information, we exploit a frame work of first-order algorithm and show that its sample complexity only linearly depends on the problem dimension implying that it is optimal in terms of depending on both problem dimension and target accuracy.  \vspace{1cm}

\section*{Appendix (Supplement)}
\appendix

\section{General Auxiliary Results}
In the following, we provide some technical tools and results that have been used in the paper proofs.

By definition \eqnok{main_prob_st}, evaluating both $\psi(x)$ and its gradient $\nabla \psi(x)$ requires explicit access to $y^(x)$. In particular, the following results hold, derived directly from applying the chain rule to the optimality condition $\nabla_y g(\bar x, y^*(\bar x)) = 0$.     
\begin{lemma}[\cite{GhadWang18}] 
\label{def_Mxy}Suppose that Assumption~\ref{fg_assumption} holds.
\begin{itemize}
\item [a)] For any $\bar x \in X$, $y^*(\bar x)$ is unique, differentiable, and 
\beq\label{grad_ystar}
\nabla y^*(\bar x) = -  \left[\nabla_{yy}^2 g(x,y^*(\bar x))\right]^{-1}\nabla_{xy}^2 g(x,y^*(\bar x))^\top.
\eeq

\item [b)] For any $\bar x \in X$:
\beq\label{grad_f2}
\nabla \psi(\bar x) = \nabla_x f(\bar x, y^*(\bar x))- \nabla_{xy}^2 g(x,y^*(\bar x)) \left[\nabla_{yy}^2 g(x,y^*(\bar x))\right]^{-1} \nabla_y f(\bar x, y^*(\bar x)).
\eeq
\end{itemize}
\end{lemma}
Clearly, $y^*(x)$ is inaccessible unless the inner problem admits a closed-form solution, creating computational difficulties for employing gradient-based algorithms in the BLP setting. To address this, we introduce an estimator that approximates $\nabla \psi(x)$, motivated by the formulation in \eqref{grad_f2}. Specifically, for any $x \in X$ and $y \in \mathbb{R}^m$, we define
\beq
\bar \nabla (f,g)(x,y) := \nabla_x f(x,y)- \nabla_{xy}^2 g(x,y)\left[\nabla_{yy}^2 g(x,y)\right]^{-1}\nabla_y f(x,y).\label{grad_f1}
\eeq

\begin{lemma}[\cite{GhadWang18}]\label{grad_f_error}
Under Assumption~\ref{fg_assumption}, for any $x \in X$ and $y \in \bbr^m$:
\begin{itemize}
\item [a)] The error of the gradient estimator $\bar \nabla f(x, y)$ depends on the quality of the approximate solution to the inner problem, i.e.,
\beq \label{def_grad_error}
\|\bar \nabla (f,g)(x, y) - \nabla \psi(x)\| \le C_1 \|y^*(x)-y\|,
\eeq
where $C_1>0$ is a constant depending on the problem parameters such as Lipschitz constants of functions $f$, $g$, and their gradients.

\item [b)] The optimal solution of the inner problem, $y^*(x)$, is Lipschitz continuous in $x$.

\item [c)] The gradient of $ \psi(x)$ is Lipschitz continuous.
\end{itemize}

\end{lemma}

\begin{proposition}[\cite{AghaGhad25}]\label{zeroth-order bilevel}
Consider the bilevel optimization problem \eqnok{main_prob_st} and its Gaussian smooth approximation \eqref{main_prob_zst}. If $f\in\mathcal{C}^1(X\times \R^m;L_{1,f})$ and $g\in\mathcal{C}^1(X\times \R^m;L_{1,g})$, then
\begin{align*}
   & \|y_{\eta_2,\mu_2}^*(x)-y^*(x)\|^2 \leq \frac{2 L_{1,g}}{\lambda_g }\left(\eta_2^2n + \mu_2^2 m\right),\\ 
    &\|\nabla \psi_{\eta,\mu}(x) - \nabla \psi(x)\| \le L_{1, f} \sqrt{\frac{2 L_{1,g}}{\lambda_g }\left(\eta_2^2n + \mu_2^2 m\right)}\nonumber \\
    & \qquad \qquad \qquad \quad  + \frac{L_{1,f}}{2}\left(\eta_1(n+3)^{\frac{3}{2}} + \frac{\mu_1^2}{\eta_1} m n^{\frac{1}{2}} + \frac{\eta_1^2}{\mu_1} n m^{\frac{1}{2}}+ \mu_1(m+3)^{\frac{3}{2}} \right).
\end{align*}
\end{proposition}

\begin{proposition}[\cite{AghaGhad25}]\label{propMomentBounds}
Consider $x\in\mathbb{R}^n$, $y\in \mathbb{R}^m$. For $q: \mathbb{R}^{n+m}\to \mathbb{R}$, where $q\in \mathcal{C}^1(\mathbb{R}^{n+m};L_{1,q})$, consider the stochastic gradient $\tilde \nabla_x q_{\eta,\mu}(x,y)$ in \eqref{stoch:gradx_minibatch} where $Q$ is replaced with $q$ and $N=1$. Then
\begin{align*}
\EE  \|\tilde \nabla_x q_{\eta,\mu}(x,y) \|^2  \leq L_{1,q}^2\left(\eta^2(n+6)^3+\frac{\mu^4}{\eta^2}n(m+4)^2 \right) &+  4(n+2) \|\nabla_x  q(x, y)\|^2 \\ &+  \frac{4\mu^2}{\eta^2}n\|\nabla_y  q(x, y)\|^2.
\end{align*}
\end{proposition}

\begin{proposition}[\cite{AghaGhad25}]\label{hessMomentBounds}
Consider a given $\theta\in \R^n$, $x\in\mathbb{R}^n$, $y\in \mathbb{R}^m$, and $q: \mathbb{R}^{n+m}\to \mathbb{R}$, where $q\in \mathcal{C}^2(\mathbb{R}^{n+m};L_{2,q})$.
\begin{itemize}
    \item [a)] The stochastic Hessian block \eqref{stochHessxx_minibatch} where $Q$ is replaced with $q$ and $N=1$, obeys
\begin{align*}\notag 
\EE_{u,v|\theta}\Big\|  \tilde \nabla_{xx}^2 q_{\eta,\mu}(x,y) \theta  \Big\|^2 & \leq  ~2L_{2,q}^2\left( 2\eta^2(n+16)^4 + \frac{\mu^6}{\eta^4}(m+6)^3 (n+3)\right)\|\theta\|^2\\ \notag &~~~~+ \bigg( \frac{15}{2} (n+6)^2\left\|\nabla_{xx}^2 q(x,y)\right\|_F^2 \\&~~~~~~~~+ \frac{3\mu^2}{\eta^2}(3n+13)\|\nabla_{xy}^2q(x,y)\|_F^2 \\& ~~~~~~~~+ \frac{3\mu^4}{2\eta^4}(m+2)(n+3)\|\nabla_{yy}^2q(x,y)\|_F^2\bigg)\|\theta\|^2.
\end{align*}

    \item [b)] The stochastic Hessian block \eqref{stochHessxy_minibatch} where $Q$ is replaced with $q$ and $N=1$, obeys
\begin{align*}\notag 
\EE_{u,v|\theta}\Big\|  \tilde \nabla_{xy}^2 q_{\eta,\mu}(x,y) \theta  \Big\|^2  &\leq  ~8L_{2,q}^2\left[ \frac{\eta^4}{\mu^2}(n+8)^4 + \frac{2\mu^4}{\eta^2} n(m+12)^3\right]\|\theta\|^2\\ \notag &~~~~ + \bigg( \frac{6\eta^2}{\mu^2}(n+4)(n+2)\|\nabla^2_{xx} q\|_F^2 \\&~~~~~~~~+ 36(n+2)\|\nabla^2_{xy}q(x,y)\|_F^2 \\& ~~~~~~~~+ \frac{30\mu^2}{\eta^2}n(m+2)\|\nabla^2_{yy}q(x,y)\|_F^2\bigg)\|\theta\|^2.
\end{align*}
\end{itemize}

\end{proposition}

    \begin{lemma}\label{lemma:4thMoment}
    Consider $x\in\mathbb{R}^n$, $y\in \mathbb{R}^m$. For $q: \mathbb{R}^{n+m}\to \mathbb{R}$, where $q\in \mathcal{C}^1(\mathbb{R}^{n+m};L_{1,q})$, consider the stochastic gradient $\tilde \nabla_x q_{\eta,\mu}(x,y)$ defined as
    \begin{align*}
     \tilde \nabla_x q_{\eta,\mu}(x,y)  =  \frac{q(x+\eta u, y+\mu v) - q(x, y)}{\eta} \ u,
\end{align*}
    where $u$ and $v$ are independent standard normal vectors. Then
\begin{align*}
\EE \|\tilde \nabla_x q_{\eta,\mu}(x,y) \|^4 \! \leq&~ 4L_{1,q}^4\left(\eta^4(n+12)^6+\frac{\mu^8}{\eta^4}(n+4)^2(m+8)^4 \right)\\ &+ 192\left(n^2 + 10n + 24\right)\|\nabla_x  q(x, y)\|^4 \\& + 192(n+4)^2\frac{\mu^4}{\eta^4}\|\nabla_y  q(x, y)\|^4.
\end{align*}
Similarly, for
\begin{align*}
     \tilde \nabla_y q_{\eta,\mu}(x,y) & =  \frac{q(x+\eta u, y+\mu v) - q(x, y)}{\mu} \ v,
\end{align*}
we have
\begin{align*}
\EE  \|\tilde \nabla_y q_{\eta,\mu}(x,y) \|^4 \! \leq&~ 4L_{1,q}^4\left(\mu^4(m+12)^6+\frac{\eta^8}{\mu^4}(m+4)^2(n+8)^4 \right)\\ &+192(m+4)^2\frac{\eta^4}{\mu^4}\|\nabla_x  q(x, y)\|^4\\ &+ 192\left(m^2 + 10m + 24\right)\|\nabla_y  q(x, y)\|^4.
\end{align*}
\end{lemma}

\begin{proof}
Notice that
\begin{align*}
    \mathbb{E}\|\tilde \nabla_x q_{\eta,\mu}(x,y) \|^4 &= \frac{1}{\eta^4}\mathbb{E}\left[\left( q(x+\eta u, y+\mu v) - q(x, y)\right)^4\|u\|^4 \right].
\end{align*}
The function-dependent term can be bounded as
\begin{align*}
    ( q(x\!&+\!\eta u, y\!+\!\mu v) - q(x, y))^4 \\&= \left( q(x\!+\!\eta u, y\!+\!\mu v) - q(x, y) - (\eta u,\mu v)^\top \nabla  q(x, y) + (\eta u,\mu v)^\top \nabla  q(x, y) \right)^4\\ &\leq 8\left( q(x\!+\!\eta u, y\!+\!\mu v) - q(x, y) - (\eta u,\mu v)^\top \nabla  q(x, y)\right)^4 + 8\left((\eta u,\mu v)^\top \nabla  q(x, y)\right)^4\\&\leq 
    \frac{L_{1,q}^4}{2}\left(\eta^2\|u\|^2+\mu^2\|v\|^2 \right)^4+ 8\left( \eta u^\top \nabla_x  q(x, y)+ \mu v^\top \nabla_y  q(x, y) \right)^4\\ 
    &\leq 4 L_{1,q}^2\left(\eta^8\|u\|^8+\mu^8\|v\|^8 \right)+ 64\left(\eta^4 \left(u^\top \nabla_x  q(x, y)\right)^4 + \mu^4\left( v^\top \nabla_y  q(x, y) \right)^4\right),
\end{align*}
   which implies
\begin{align}\notag 
      \|\tilde \nabla_x q_{\eta,\mu}(x,y) \|^4\leq& 4L_{1,q}^2\left[\eta^4\|u\|^{12}+\frac{\mu^8}{\eta^4}\|v\|^8\|u\|^4 \right]\\&+ 64\left( \left(u^\top \nabla_x  q(x, y)\right)^4 + \frac{\mu^4}{\eta^4}\left( v^\top \nabla_y  q(x, y) \right)^4\right)\left(u^\top u\right)^2.\label{eq:expansionGradq}
\end{align}
Using Lemma 1 from \cite{nesterov2017random}, the expectation of the first term on the right side of \eqref{eq:expansionGradq} can be bounded as
\begin{align}
       L_{1,q}^4\EE\left[\eta^4\|u\|^{12}+\frac{\mu^8}{\eta^4}\|v\|^8\|u\|^4 \right] \leq L_{1,q}^4\left(\eta^4(n+12)^6+\frac{\mu^8}{\eta^4}(n+4)^2(m+8)^4 \right). \label{eq:expansionGradq2}
\end{align}
The second term expectation in \eqref{eq:expansionGradq} can be written as
\begin{align}\notag 
   \EE\bigg[ &\Big(u^\top \nabla_x  q(x, y)\Big)^4 \left(u^\top u\right)^2 + \frac{\mu^4}{\eta^4}\Big( v^\top \nabla_y  q(x, y) \Big)^4 \left(u^\top u\right)^2 \bigg] \\\notag &= \EE\left[\left(u^\top \nabla_x  q(x, y) \nabla_x  q(x, y)^\top u u^\top u\right)^2\right] \\\notag &~~~~~~+ \frac{\mu^4}{\eta^4}\EE\left[\left(v^\top \nabla_y  q(x, y) \nabla_y  q(x, y)^\top v\right)^2\right]\EE\left[ \|u\|^4\right]\notag \\& = 3(n^2 + 10n + 24)\|\nabla_x  q(x, y)\|^4 + 3\frac{\mu^4}{\eta^4}\|\nabla_y  q(x, y)\|^4(n+4)^2,\label{eq:expansionGradq3}
\end{align}
where in the second equality we used the 8-th moment result of \cite{magnus1978moments} (see Lemma 6.2 therein) that for symmetric matrices $A$ and $B$:
\begin{align}\label{cookbook:result:8thorder}\notag 
\EE&\left[ \left(u^\top A u u^\top B u\right)^2 \right]= \\& ~\notag \left(\tr(A)\tr(B)\right)^2 + 16\left[ \tr(A)\tr\left(AB^2\right) + \tr(B)\tr\left(A^2B\right)\right]\\& 
\notag  + 4\left[ \tr\left(A^2\right)\tr\left(B^2\right) + 2\left(\tr\left(AB\right)\right)^2\right]\\ 
&\notag + 2\left[ \left(\tr(A)\right)^2\tr\left(B^2\right) + 4\tr(A)\tr(B)\tr(AB) + \left(\tr(B)\right)^2\tr\left(A^2\right) \right]\\&+16\left[ \left(\tr(AB)\right)^2 + 2\tr\left( A^2 B^2\right)\right],
\end{align}
and the 4-th moment result:
\begin{equation}\label{cookbook:result:4thorder}
\EE u^\top A u u^\top B u = 2\tr(AB) + \tr(A)\tr(B).
\end{equation}
Using \eqref{eq:expansionGradq3} and \eqref{eq:expansionGradq2} in \eqref{eq:expansionGradq} yields 
\begin{align*}
\EE \|\tilde \nabla_x q_{\eta,\mu}(x,y) \|^4 \! \leq&~ 4L_{1,q}^4\left(\eta^4(n+12)^6+\frac{\mu^8}{\eta^4}(n+4)^2(m+8)^4 \right)\\ &+ 192\left(n^2 + 10n + 24\right)\|\nabla_x  q(x, y)\|^4 \\& + 192(n+4)^2\frac{\mu^4}{\eta^4}\|\nabla_y  q(x, y)\|^4.
\end{align*}
\end{proof}

\section{Section \ref{sec:intro} Auxiliary Results}
\subsection{Proof of Proposition \ref{propNestApprox_stch}}\label{SMProof:1}
    
	\paragraph{Part (a)}\begin{proof}
		Noting \eqref{stoch:gradx_minibatch}, $\EE \left\|\tilde \nabla_x Q_{\eta,\mu}\left(x,y,\{\zeta_i\}_{i=1}^N\right) - \nabla_x q_{\eta,\mu}(x,y)\right\|^2$ can be written as 
		\begin{equation}\label{eq:nablaxtildeVar}
			  \EE \left\|\tilde \nabla_x Q_{\eta,\mu}\left(x,y,\{\zeta_i\}_{i=1}^N\right)\right\|^2 - \left\|\nabla_x q_{\eta,\mu}(x,y)\right\|^2\leq \EE \left\|\tilde \nabla_x Q_{\eta,\mu}\left(x,y,\{\zeta_i\}_{i=1}^N\right)\right\|^2. 
		\end{equation}
		It all amounts to bound the right-hand side of \eqref{eq:nablaxtildeVar}. Because of the independence among the summand terms in $\tilde \nabla_x Q_{\eta,\mu}\left(x,y,\{\zeta_i\}_{i=1}^N\right)$, 
		\begin{align*}
			\EE \left\|\tilde \nabla_x Q_{\eta,\mu}\left(x,y,\{\zeta_i\}_{i=1}^N\right)\right\|^2= \frac{1}{N}\EE \left\|\tilde \nabla_x Q_{\eta,\mu}\left(x,y,\zeta\right)\right\|^2,
		\end{align*}
		where 
		\begin{align}\notag 
			\tilde \nabla_x Q_{\eta,\mu}(x,y,\zeta)  =  \frac{Q(x+\eta u, y+\mu v,\zeta) - Q(x, y,\zeta)}{\eta} \ u.
		\end{align}
		Applying Proposition \ref{propMomentBounds} we get
		\begin{align*}
			\frac{1}{N}\EE_{u,v} \left\|\tilde \nabla_x Q_{\eta,\mu}\left(x,y,\zeta\right)\right\|^2&\leq \frac{L_{1,Q}^2}{N}\left(\eta^2(n+6)^3+\frac{\mu^4}{\eta^2}n(m+4)^2 \right)\\ &~~~+  \frac{4(n+2)}{N} \|\nabla_x  Q(x, y,\zeta)\|^2\\& ~~~+  \frac{4\mu^2}{\eta^2N}n\|\nabla_y  Q(x, y,\zeta)\|^2.
		\end{align*}
		Now taking an expectation of both sides with respect to $\zeta$ gives
		\begin{align*}
			\frac{1}{N}\EE \left\|\tilde \nabla_x Q_{\eta,\mu}\left(x,y,\zeta\right)\right\|^2& \leq \frac{L_{1,Q}^2}{N}\left(\eta^2(n+6)^3+\frac{\mu^4}{\eta^2}n(m+4)^2 \right) \\&~~~+  \frac{4(n+2)}{N}\left(\sigma_{1,Q}^2+ \|\nabla_x  q(x, y)\|^2\right) \\&~~~+  \frac{4\mu^2}{\eta^2N}n\left(\sigma_{1,Q}^2+\|\nabla_y  q(x, y)\|^2\right).
		\end{align*}
	\end{proof}
	\paragraph{Part (b)}\begin{proof}
		Using a similar independence argument as part (a) we get
		\begin{align}\notag 
			\EE_{\{u_i,v_i,\zeta_i\}_{i=1}^N\mid \theta } &\left\|\tilde \nabla_{xy}^2 Q_{\eta,\mu}\left(x,y,\{\zeta_i\}_{i=1}^N\right)\theta -  \nabla_{xy}^2 q_{\eta,\mu}(x,y)\theta  \right\|^2 \\&\leq   \frac{1}{N}\EE_{u,v,\zeta \mid \theta } \left\|\tilde \nabla_{xy}^2 Q_{\eta,\mu}\left(x,y,\zeta\right)\theta  \right\|^2\label{1OverN}
		\end{align}
		where 
		\begin{align}\notag 
			\tilde \nabla_{xy}^2 Q_{\eta,\mu}(x,y,\zeta)  =  u v^\top  \frac{Q(x+\eta u, y+\mu v,\zeta) + Q(x-\eta u,y-\mu v,\zeta)-2Q(x,y,\zeta )}{\eta\mu}.
		\end{align}
		Applying Proposition \ref{hessMomentBounds} we get
		\begin{align*}\notag 
			\EE_{u,v|\zeta, \theta}\Big\|  \tilde \nabla_{xy}^2 Q_{\eta,\mu}(x,y) \theta  \Big\|^2  &\leq  ~8L_{2,Q}^2\left[ \frac{\eta^4}{\mu^2}(n+8)^4 + \frac{2\mu^4}{\eta^2} n(m+12)^3\right]\|\theta\|^2\\ \notag &~~~~+ \bigg( \frac{6\eta^2}{\mu^2}(n+4)(n+2)\|\nabla^2_{xx} Q(x,y,\zeta)\|_F^2 \\&~~~~~~~~~~+ 36(n+2)\|\nabla^2_{xy}Q(x,y,\zeta)\|_F^2 \\& ~~~~~~~~~~+ \frac{30\mu^2}{\eta^2}n(m+2)\|\nabla^2_{yy}Q(x,y,\zeta )\|_F^2\bigg)\|\theta\|^2.
		\end{align*}
		Taking an expectation of both sides with respect to $\zeta$ we get: 
		\begin{align*}\notag 
			\EE_{u,v,\zeta| \theta}\Big\|  \tilde \nabla_{xy}^2 Q_{\eta,\mu}(x,y) \theta  \Big\|^2  &\leq  ~8L_{2,Q}^2\left[ \frac{\eta^4}{\mu^2}(n+8)^4 + \frac{2\mu^4}{\eta^2} n(m+12)^3\right]\|\theta\|^2\\ \notag &~~~~+ \bigg( \frac{6\eta^2}{\mu^2}(n+4)(n+2)\left(\sigma_{2,Q}^2 + \|\nabla^2_{xx} q(x,y)\|_F^2\right) \\&~~~~~~~~~+ 36(n+2)\left(\sigma_{2,Q}^2 + \|\nabla^2_{xy}q(x,y)\|_F^2\right) \\& ~~~~~~~~~+ \frac{30\mu^2}{\eta^2}n(m+2)\left(\sigma_{2,Q}^2 + \|\nabla^2_{yy}q(x,y )\|_F^2\right)\bigg)\|\theta\|^2,
		\end{align*}
		combining which with \eqref{1OverN} completes the proof. 
\end{proof}

\section{Section \ref{near_opt} Auxiliary Results}
\subsection{Proof of Lemma \ref{sgd:hessInv}}
\label{sec:2proof1}
\begin{proof}
		By \cite[Theorem 2]{AghaGhad25} and setting $\gamma={\cal O}(\epsilon/V_H)$, we have 
		\beq\label{proof_SG0}
		\|\EE [z_T] - \bar z\|^2 \le \EE \| z_T - \bar z\|^2  \leq (1-\gamma \lambda_g)^T\|z_0 - \bar z\|^2 + {\cal O} \left(\frac{\epsilon \bar V}{V_H}\right),
		\eeq
		where
		\[
		V_H={\cal O}\left(\mu_2^2 m( m+n)^3 + m(m+n)^2 \right), \bar V = {\cal O} \left(V_H+\mu_1^2 m^3+m \right).
		\]
		As long as $\mu_2 \le 1/\sqrt{m+n}$, we have $V_H={\cal O}(m(m+n)^2)$ and so $\gamma = {\cal O}(\epsilon/m(m+n)^2)$. Hence, we get 
		$\bar V/V_H = {\cal O}(1)$ when $\mu_1 \le 1/m$. By choosing $T = \frac{1}{\gamma}\log(1/\epsilon)$, we also get $(1-\gamma \lambda_g)^T\|z_0 - \bar z\|^2 ={\cal O}(\epsilon)$ which together with \eqnok{proof_SG0} imply the result.
		
	\end{proof}

\subsection{Proof of Lemma \ref{lemma_sgd}}\label{sec:2proof2}
	\begin{proof}
		By \cite[Lemma 4]{AghaGhad25}, we have 
		\begin{align*}
			\EE &\|y_{t_k} - y_{\eta_2,\mu_2}^*(x_k) \|^2  \le 8 \epsilon_k \left[\|y_0-y^*(x_k)\|^2 + 8(m+4)\sigma_{1,G}^2 \right]\nn \\
			&  + 4\mu_2^2 \left[8\epsilon_k (L_{1,G}^2+L_{1,g}^2)(m+4)^4 + \frac{ L_{1,g}}{\lambda_g } \left((3+4\epsilon_k)m+ \frac{\eta_2^2n}{\mu_2^2}\right)\right]
		\end{align*}
		when   
		\[
		\beta_t = \beta(\epsilon_k) =  \min\left\{\frac{1}{8(m+4) L_{1,G}}, \epsilon_k\lambda_g, \frac{1}{\lambda_g}\right\}, \qquad t_k \ge  \left\lceil \frac{1}{\lambda_g \beta(\epsilon_k)}\log\frac{1}{\epsilon_k}\right\rceil \qquad  \forall t
		\]
		for some $\epsilon_k >0$. The result then immediately follows by choosing $\epsilon_k = \epsilon/m$ and according to the choice of parameters in \eqnok{alpha_beta_st}.
		
	\end{proof}
	
\section{Section \ref{full_opt} Auxiliary Results}
\subsection{Proof of Lemma \ref{lemma:yz:diffs} }\label{sec:3proof1}
\begin{proof}
    The first inequality comes from applying Lemma 3.2 in \cite{kwon2023fully}. For the second relation, we note that
    \begin{align*}
        \frac{1}{\lambda} \nabla_y f(x, y^*(x)) + \nabla_y g(x, y^*(x)) = 0, \\
        \nabla_y g(x, z^*(x)) = 0.
    \end{align*}
    Arranging the two equations, we have
    \begin{align*}
        &\nabla_2 g(x, y^*(x)) - \nabla_y g(x, z^*(x)) \\
        &= \nabla_{yy}^2 g(x, z^*(x)) (y^*(x) -z^* (x) ) + \mathcal{O} \left( L_{2,g} \| y^*(x) -z^* (x) \|^2 \right) \\
        &= \nabla_{yy}^2 g(x, z^*(x)) (y^*(x) -z^* (x) ) + \mathcal{O} \left( \frac{L_{2,g} L_{0,f}^2}{\lambda_g^2 \lambda^2}  \right) \\
        &= -\frac{1}{\lambda} \nabla_y f(x, y^*(x)).
    \end{align*}
    With $\nabla_{yy}^2 g(x, z^*(x)) \succeq \lambda_g I$, we have the lemma.
\end{proof}

\subsection{Proof of Lemma \ref{lemma:zeroth_first_order_approx_bias}}\label{sec:3proof2}
\begin{proof}
    The first part comes from \cite[Proposition 6]{AghaGhad25}. For the second part, we first state a variant of Lemma~\ref{lemma:yz:diffs} with Gaussian-smoothing whose proof is similar to that of Lemma~\ref{lemma:yz:diffs} and hence, we skip it.
    \begin{lemma} \label{lemma:yz_diff_mu}
        For any given $x$, we have
        \begin{align*}
            &\|y_\mu^*(x) - z_{\mu}^*(x)\| \le \frac{2 L_{0,f}}{\lambda_g \lambda}, \\
            &y^*_\mu(\bar{x}) - z^*_\mu(x) = -\frac{1}{\lambda} \nabla_{yy}^2 g_{0,\mu}(x,z_\mu^*(x))^{-1} \nabla_y f_{0,\mu}(x, y_\mu^*(x))  + \mathcal{O}\left( \frac{L_{2,g} L_{0,f}^2}{\lambda_g^3 \lambda^2} \right).
        \end{align*}
    \end{lemma}
    We first decompose the inequality such that
    \begin{align*}
        &\| \nabla \mathcal{L}_{\eta,\mu}^*(x) - \nabla \mathcal{L}^*(x) \| \\
          &\le \|\nabla_x f_{\eta,0} (x,y^*_\mu(x)) - \nabla_x f (x,y^*(x)) \| \\
          & + \lambda \|(\nabla_x g(x,y^*(x)) -  \nabla_x g(x,z^*(x))) - (\nabla_x g_{\eta,0} (x,y_\mu^*(x)) - \nabla_x g_{\eta,0}(x,z_\mu^*(x))) \|  \\
        &\le \|\nabla_x f_{\eta,0} (x,y^*_\mu(x)) - \nabla_x f (x,y^*(x)) \| + \mathcal{O}(\lambda L_{2,g}) (\|y^*(x) - z^*(x)\|^2 + \|y_{\mu}^*(x) - z_{\mu}^*(x)\|^2) \\
        & + \lambda \| \nabla_{xy}^2 g(x,z^*(x)) (y^*(x) - z^*(x)) -  \nabla_{xy}^2 g(x,z_{\mu}^*(x)) (y_{\mu}^*(x) - z_{\mu}^*(x)) \| \\
        &\le \|\nabla_x f_{\eta,0} (x,y^*_\mu(x)) - \nabla_x f (x,y^*(x)) \| + \mathcal{O}\left( \frac{L_{2,g} L_{0,f}^2}{\lambda_g^2 \lambda} \right) \\
        &\quad + \lambda \cdot \|\nabla_{xy}^2 g_{\eta,0} (x,y_\mu^*(x)) - \nabla_{xy}^2 g(x,y_\mu^*(x))\| \cdot \|y^*_\mu(x)-z_\mu^*(x)\| \\ 
        &\quad + L_{1,g} \cdot \| \nabla_{yy}^2 g(x, z^*(x))^{-1} \nabla_y f(x,y^*(x)) - \nabla_{yy}^2 g_{0,\mu}(x, z_\mu^*(x))^{-1} \nabla_y f_{0,\mu}(x,y_\mu^*(x))\|,
    \end{align*}
    where the last step applied Lemma \ref{lemma:yz_diff_mu}. To bound the first term,
    \begin{align*}
        &\|\nabla_x f_{\eta,0} (x,y^*_\mu(x)) - \nabla_x f (x,y^*(x)) \| \\
        &\le \|\nabla_x f_{\eta,0} (x,y^*(x)) - \nabla_x f (x,y^*(x)) \| + L_{1,f} \|y_\mu^*(x) - y^*(x)\| \\
        &\le \mathcal{O}\left(L_{1,f} \eta n^{3/2} + L_{1,f} \mu \sqrt{\frac{m L_{1,g}}{\lambda_g}}\right). 
    \end{align*}
    For the remaining terms, we require the following lemma:
    \begin{lemma}
        For any $x \in X$ and $y \in \mathbb{R}^{m}$, we have
        \begin{align*}
            \|\nabla_{xy}^2 g_{\eta,0} (x,y) - \nabla_{xy}^2 g(x,y)\| &\le \mathcal{O} (L_{2,g} \eta n^{3/2}), \\
            \|\nabla_{yy}^2 g_{0,\mu} (x,y) - \nabla_{yy}^2 g(x,y)\| &\le \mathcal{O} (L_{2,g} \mu m^{3/2}). 
        \end{align*}
        \label{lemma:second_deriv_approx_error}
    \end{lemma}
    The proof of Lemma \ref{lemma:second_deriv_approx_error} is deferred to the end. Using this lemma, note that
    \begin{align*}
        &\| \nabla_{yy}^2 g(x, z^*(x))^{-1} \nabla_y f(x,y^*(x)) - \nabla_{yy}^2 g_{0,\mu}(x, z_\mu^*(x))^{-1} \nabla_y f_{0,\mu}(x,y_\mu^*(x))\| \\
        &\le \|\nabla_{yy}^2 g(x, z^*(x))^{-1}\| \|\nabla_y f_{0,\mu} (x,y^*_\mu(x)) - \nabla_y f (x,y^*(x)) \| \\
        &\quad + \|\nabla_{yy}^2 g(x, z^*(x))^{-1}\| \|\nabla_{yy}^2 g(x, z^*(x)) - \nabla_{yy}^2 g_{0,\mu}(x, z_\mu^*(x))\| \|\nabla_{yy}^2 g_{0,\mu}(x, z_\mu^*(x))^{-1}\| \|\nabla_y f (x,y^*(x)) \| \\
        &\le \frac{1}{\lambda_g} \cdot \mathcal{O}\left(L_{1,f} \mu m^{3/2} + L_{1,f} \mu \sqrt{\frac{m L_{1,g}}{\lambda_g}}\right) + \frac{L_{0,f}}{\lambda_g^2} \cdot \mathcal{O}\left( L_{2,g} \mu m^{3/2} + L_{2,g} \mu \sqrt{\frac{m L_{1,g}}{\lambda_g}} \right),
    \end{align*}
    where we used
    \begin{align*}
        &\|\nabla_{yy}^2 g_{0,\mu} (x,z^*(x)) - \nabla_{yy}^2 g(x,z_\mu^*(x))\| \\
        &\le \|\nabla_{yy}^2 g_{0,\mu} (x,z_\mu^*(x)) - \nabla_{yy}^2 g(x,z_\mu^*(x))\| + \|\nabla_{yy}^2 g (x,z^*(x)) - \nabla_{yy}^2 g(x,z_\mu^*(x))\| \\
        &\le \mathcal{O} \left(L_{2,g} \mu m^{3/2} + L_{2,g} \mu  \sqrt{ \frac{m L_{1,g}}{\lambda_g}} \right). 
    \end{align*}
    Combining all, we get
    \begin{align*}
        &\| \nabla \mathcal{L}_{\eta,\mu}^*(x) - \nabla \mathcal{L}^*(x) \| \\
        &\le \mathcal{O} \left(\frac{L_{1,g} L_{1,f} }{\lambda_g}(\eta n^{3/2} + \mu m^{3/2} + \mu (m L_{1,g}/\lambda_g)^{1/2} ) \right) \\
        &\quad + \mathcal{O} \left( \frac{L_{0,f} L_{1,g} L_{2,g}  }{\lambda_g^2} (\mu m^{3/2} + \mu (m L_{1,g}/\lambda_g)^{1/2}) \right) \\
        &\quad + \mathcal{O} \left(\frac{L_{0,f} L_{2,g} \eta n^{3/2}}{\lambda_g}\right) + \mathcal{O}\left( \frac{L_{0,f}^2 L_{2,g}}{\lambda_g^2 \lambda} \right).
    \end{align*}

\paragraph{Proof of Lemma \ref{lemma:second_deriv_approx_error}} 
We first note that
\begin{align*}
    \nabla_y (\nabla_x g_{\eta,0}(x,y)) &= \nabla_y \left(\EE[\tilde \nabla_x G(x,y; \zeta)] \right) = \nabla_y \left(\EE \left[\frac{G(x+\eta u, y;\zeta) - G(x, y; \zeta)}{\eta} u \right]\right) \\
    &= \nabla_y \left(\EE \left[\frac{g(x+\eta u, y) - g(x, y)}{\eta} u \right]\right) \\& =  \EE \left[\frac{\nabla_y g(x+\eta u, y) - \nabla_y g(x, y)}{\eta} u^\top \right], 
\end{align*}
where the last equality follows from interchangebility of derivative and expectation since $g$ is smooth and expectation is taken over Gaussian distribution~\cite{folland1999real}.
Therefore, we get
\begin{align*}
    &\nabla_{yx}^2 g_{\eta,0} (x,y) - \nabla_{yx}^2 g(x,y) \\
    &= \EE \left[\frac{\nabla_y g(x+\eta u, y) - \nabla_y g(x, y) }{\eta} u^\top \right] - \nabla_{yx}^2 g(x,y) \\
    &= \EE \left[\nabla_{yx}^2 g(x, y) u u^\top \right] - \nabla_{yx}^2 g(x,y) + \mathcal{O}(L_{2,g} \eta \|u\|^{3})  \\
    &= \mathcal{O} (L_{2,g} \eta n^{3/2}).
\end{align*}
The second inequality for $\nabla^{2}_{yy}g$ can be obtained similarly.
\end{proof}
\subsection{Proof of Lemma \ref{lemma:inexact_first_bias}}\label{sec:3proof3}
\begin{proof}
    For simplicity, we omit the subscript $k$ here whenever the context is clear, and denote $y_\mu^*(x), z_\mu^*(x)$ by $y^*, z^*$, respectively. Note the following:
    \begin{align*}
        &\|\nabla_x f_\eta(x, y) - \nabla_x f_\eta(x,y^*)\| \le L_{1,f} \|y-y^*\|, \\
        & \nabla_x g_{\eta,0}(x,y^*) - \nabla_x g_{\eta,0}(x,z^*) = \nabla_{xy}^2 g_{\eta,0}(x,y^*) (y^*-z^*) + \mathcal{O}(L_{2,g} \|y^*-z^*\|^2), \\
        &\nabla_x g_{\eta,0}(x,y) - \nabla_x g_{\eta,0}(x,z) = \nabla_{xy}^2 g_{\eta,0}(x,y) (y-z) + \mathcal{O}(L_{2,g} \|y-z\|^2).
    \end{align*}
    Then again, we have
    \begin{align*}
        &\| \nabla_{xy}^2 g_{\eta,0}(x,y) (y-z) - \nabla_{xy}^2 g_{\eta,0}(x,y^*) (y^*-z^*)\|\\
        &\le \|\nabla_{xy}^2 g_{\eta,0}(x,y^*) (y-z - q^*)\| + \| (\nabla_{xy}^2 g_{\eta,0}(x,y) - \nabla_{xy}^2 g_{\eta,0}(x,y^*)) (y^*-z^*)\| \\
        &\le L_{1,g} \|q-q^*\| + L_{2,g} \|q^*\| \|y-y^*\|. 
    \end{align*}
    Combining all inequalities, we get the lemma.
\end{proof}

\subsection{Proof of Lemma \ref{lemma:control_y_z}}\label{sec:3proof4}
\begin{proof}
    We first show the inequality for $z_t$. For simplicity, we denote $x = x_k$ and $z^* = z_\mu^* (x)$. Note that
    \begin{align*}
        \EE_{|t} \|z_{t+1} - z^*\|^2 &\le \|z_t - z^*\|^2 - 2\beta \vdot{z_t - z^*}{\nabla_z g_{0,\mu}(z_t)} + \beta^2 \EE_{|t} \|\tilde{\nabla}_z G_{0,\mu} (x, z_t; \xi)\|^2 \\
        &\le (1-\beta\lambda_g) \|z_t - z^*\|^2 + \beta^2 \cdot \texttt{Var} \left( \tilde{\nabla}_z G_{0,\mu} (x, z_t; \xi) \right) \\
        &\le (1-\beta\lambda_g) \|z_t - z^*\|^2 + \beta^2 \cdot \mathcal{O} \left( L_{1,G}^2 \mu^2 m^3 + m \sigma_{1,G}^2 \right).
    \end{align*}
    The result for $y_t$ can be obtained similarly. As we set $\lambda \ge 4 L_{1,f} / \lambda_g$, $g + \lambda^{-1} f$ is $(3\lambda_g/4)$-strongly-convex, and thus, we have
    \begin{align*}
        \EE_{|t} &\|y_{t+1} - y^*\|^2\\
        &\le \|y_t - y^*\|^2 - 2\beta \vdot{y_t - y^*}{\nabla_y g_{0,\mu}(y_t)} \\&\quad + \beta^2 \EE_{|t} \|\lambda^{-1} \tilde{\nabla}_y F_{0,\mu} (x, y_t; \xi) + \tilde{\nabla}_y G_{0,\mu} (x, y_t; \xi)\|^2 \\
        &\le (1-\beta\lambda_g) \|z_t - z^*\|^2 + \beta^2 \cdot \texttt{Var} \left( \tilde{\nabla}_z G_{0,\mu} (x, z_t; \xi) \right) + \frac{\beta^2}{\lambda^2} \cdot \texttt{Var} \left( \tilde{\nabla}_y F_{0,\mu} (x, y_t; \xi) \right) \\
        &\le (1-\beta\lambda_g) \|z_t - z^*\|^2 + \beta^2 \cdot \mathcal{O} \left( \mu^2 L_{1,G}^2 m^3 + m \sigma_{1,G}^2 + \lambda^{-2} m \sigma_{1,F}^2 \right).
    \end{align*}
\end{proof}

\subsection{Proof of Lemma \ref{lemma:yz_diff_descent}}\label{sec:3proof5}
\begin{proof}
    By the strong-convexity of $g_{\mu}$, note that
    \begin{align*}
        \vdot{y_t-z_t}{\nabla_y g_{0,\mu} (x,y_t) - \nabla_y g_{0,\mu} (x,z_t)} \ge \lambda_g \|y_t - z_t\|^2,
    \end{align*}
    due to Proposition 1 in \cite{AghaGhad25}. By unfolding the update, we have
    \begin{align*}
        \EE_{|t} &\|y_{t+1} - z_{t+1}\|^2 \\
        &= \|y_t - z_t\|^2 + \beta^2 \cdot \EE_{|t} \|\lambda^{-1} \tilde{\nabla}_y F_{0,\mu}(x,y_t) + \tilde{\nabla}_y G_{0,\mu}(x,y_t) - \tilde{\nabla}_z G_{0,\mu}(x,z_t)\|^2 \\
        &\quad - 2 \beta \vdot{y_t - z_t}{\lambda^{-1} \nabla_y f_{0,\mu} (x, y_t) + \nabla_y g_{0,\mu} (x, y_t) - \nabla_y g_{0,\mu}(x,z_t)} \\
        &\le \|y_t - z_t\|^2 + \beta^2 \lambda^{-2} \cdot \EE_{|t} \|\tilde \nabla_y F_{0,\mu}(x,y_t)\|^2 \\&\quad + \beta^2 \cdot \EE_{|t} \|\tilde \nabla_y G_{0,\mu}(x,y_t) - \tilde \nabla_y G_{0,\mu}(x,z_t))\|^2 \\
        &\quad - 2 \beta \vdot{y_t - z_t}{\nabla_y g_{0,\mu} (x, y_t) - \nabla_y g_{0,\mu}(x,z_t)} - 2 \frac{\beta}{\lambda} \vdot{y_t-z_t}{\nabla_y f_{0,\mu}(x,y_t)} \\
        &\le (1 - \lambda_g \beta) \|y_t - z_t\|^2 + \mathcal{O} \left( \frac{\beta L_{0,f}^2}{\lambda^2 \lambda_g} \right) \\&\quad + \beta^2 \cdot \mathcal{O} \left(\frac{m \sigma_{1,F}^2}{\lambda^2}  + L_{1,G}^2 \mu^2 m^3 + L_{1,G}^2 m \|y_t - z_t\|^2 \right),
    \end{align*} 
    where we note that
    \begin{align*}
        \EE_{|t} \|\tilde \nabla_y F_{0,\mu}(x,y_t)\|^2 &\le L_{0,f}^2 + m\sigma_{1,F}^2 + L_{1,F}^2 \mu^2 m^3,
    \end{align*}
    and following the same proof steps to Lemma \ref{lemma:G_diff_var}, we have
    \begin{align*}
        &\EE_{|t} \|\tilde \nabla_y G_{0,\mu}(x,y_t) - \tilde \nabla_y G_{0,\mu}(x,z_t))\|^2 \\
        &\le \|\nabla_y g_{0,\mu} (x,y_t) - \nabla_z g_{0,\mu}(x,z_t)\|^2 + \texttt{Var} (\tilde \nabla_y G_{0,\mu}(x,y_t) - \tilde \nabla_y G_{0,\mu}(x,z_t))) \\
        &\le L_{1,g}^2 \|y_t - z_t\|^2 + \mathcal{O} (\mu^2 L_{1,G}^2 m^3 + m L_{1,G}^2 \|y_t - z_t\|^2).
    \end{align*}
    
    By choosing $\beta \ll \frac{\lambda_g}{L_{1,G}^2 m}$, we have
    \begin{align*}
        \EE_{|t} \|y_{t+1} - z_{t+1}\|^2 &\le \left(1 - \frac{\lambda_g \beta}{2}\right) \|y_t - z_t\|^2 + \mathcal{O}\left(\frac{\beta L_{f,0}^2}{\lambda^2 \lambda_g} + \frac{\beta^2 m \sigma_F^2}{\lambda^2} + \beta^2 \mu^2 m^3 L_{1,G}^2 \right).
    \end{align*}
\end{proof}

\subsection{Proof of Lemma \ref{descentLemmaFourthMoment}}\label{sec:3proof6}
\begin{proof}
    As $\|u+v+w\|^4 \le 27 (\|u\|^4 + \|v\|^4 + \|w\|^4)$, we start with
    \begin{align}\notag 
         \EE  \| \tilde \nabla_y G_{\mu}^y - \tilde \nabla_y G_{\mu}^{z} \|^4  \leq ~& \frac{27}{\mu^4}\EE\left\|   \left( G(x, y + \mu v, \zeta) - G(x, y,\zeta) - \mu \vdot{\nabla_y G(x, y, \zeta)}{v}  \right) v \right \|^4  \\ \notag & + \frac{27}{\mu^4}\EE\left\|   \left( G(x, z + \mu v, \zeta) - G(x, z,\zeta) - \mu \vdot{\nabla_y G(x, z, \zeta)}{v}  \right) v \right \|^4  \\ ~& + 27 \EE \left\| \left\langle \nabla_y G(x, y, \zeta) - \nabla_y G(x, z, \zeta), v\right\rangle v\right\|^4,  \label{3:fourth_terms}
    \end{align}
    similarly to \eqref{3:terms}. Again, note that
    \begin{align*}
        \left| G(x , y + \mu v, \zeta) - G(x,  y,\zeta) - \mu \left\langle \nabla_y G(x, y, \zeta), v \right\rangle  \right| \leq \frac{\mu^2 L_{1,G} }{2}\|v\|^2, 
    \end{align*}
    and therefore
    \begin{align*}
        \EE &\left\|   \left( G(x , y + \mu v, \zeta) - G(x, y,\zeta) - \eta \vdot{\nabla_x G(x, y, \zeta)}{u}  \right) v \right \|^4 \\
        &\le \frac{\mu^8 L_{G,1}^4}{16} \EE[\|v\|^{12}] \le \mathcal{O} \left(\mu^8 L_{1,G}^4 m^6 \right),
    \end{align*}
    and we can bound the second term similarly. For the third term, we have
    \begin{align*}
        \EE& \left\| \left\langle \nabla_y G(x, y, \zeta) - \nabla_y G(x, z, \zeta), v\right\rangle v\right\|^4 \\
        &\le \EE\left[ \left| \left\langle \nabla_y G(x, y, \zeta) - \nabla_y G(x, z, \zeta), v \right\rangle \right|^8 \right]^{1/2} \EE\left[ \|v\|^8 \right]^{1/2} \\
        &\le L_{1,G}^4 m^2 \| {y} - {z}\|^4. 
    \end{align*}
    This all implies that 
    \begin{align*}
        \EE[\|\tilde \nabla_y G_{0,\mu}^t - \tilde \nabla_z G_{0,\mu}^t\|^4] &\le \mathcal{O} \left( \mu^4 L_{1,G}^4 m^6 + L_{1,G}^4 m^2 \|y_t - z_t\|^4 \right).
    \end{align*}
    On the other hand, under Assumption~\ref{assumption:fourth_moment} and Lemma~\ref{lemma:4thMoment}, we have 
    
    \begin{align*}
        &\EE \|\tilde \nabla_y F_{0,\mu}(x,y;\zeta)\|^4 \\
        &\le 4\EE \left\|\frac{F(x,y+\mu v;\zeta) - F(x,y;\zeta) - \mu v^\top \nabla_y F(x,y;\zeta)}{\mu} v \right\|^4 + 4 \EE \left\|vv^\top \nabla_y F(x,y;\zeta) \right\|^4 \\
        &\le \mathcal{O} \left(\mu^4 L_{1,F}^4 m^6  + m^2 (L_{f,0}^4 + \sigma_{1,F}^4) \right).
    \end{align*}

    To bound the fourth-order moment of differences between $y_t$ and $z_t$, we can start with 
    \begin{align*}
        \EE_{|t}\|y_{t+1} - z_{t+1}\|^4 &= \EE_{|t} (\|y_t - z_t\|^2 + \beta^2 \|\lambda^{-1} \tilde{\nabla}_y F_{0,\mu}^t + \tilde{\nabla}_y G_{0,\mu}^t - \tilde{\nabla}_z G_{0,\mu}^t \|^2 \\
        &\quad - 2\beta \vdot{y_t - z_t}{\lambda^{-1} \tilde{\nabla}_y F_{0,\mu}^t + \tilde{\nabla}_y G_{0,\mu}^t - \tilde{\nabla}_z G_{0,\mu}^t })^2 \\
        &\le \|y_t - z_t\|^4 + \beta^4 \cdot \mathcal{O} (\lambda^{-4} n^2 \sigma_{F}^4 + L_{g,2}^4 n^2 \|y_t-z_t\|^4 + \eta^4 L_{g,1}^4 n^6) \\ 
        &\quad + 4\beta^3\|y_t-z_t\|  \EE_{|t} [\|\lambda^{-1} \tilde{\nabla}_y F_{0,\mu}^t + \tilde{\nabla}_y G_{0,\mu}^t - \tilde{\nabla}_z G_{0,\mu}^t\|^3] \\
        &\quad + 6\beta^2 \|y_t-z_t\|^2 \EE_{|t} [\|\lambda^{-1} \tilde{\nabla}_y F_{0,\mu}^t + \tilde{\nabla}_y G_{0,\mu}^t - \tilde{\nabla}_z G_{0,\mu}^t\|^2] \\
        &\quad - 4\beta \|y_t-z_t\|^2 \vdot{y_t - z_t}{\nabla_y g_{0,\mu}(x,y_t) - \nabla_z g_{0,\mu}(x,z_t)} \\
        &\quad - 4\beta \|y_t-z_t\|^2 \vdot{y_t - z_t}{\lambda^{-1} \nabla_y f_{0,\mu}(x,y_t)}.
    \end{align*}
    Note that for all $a,b \ge 0$, by Young's inequality,
    \begin{align*}
        ab^3 &\le \frac{\lambda_g^3 a^4}{4} + \frac{3b^4}{4\lambda_g}, \\
        a^2b^2 &\le \frac{\lambda_g a^4}{2} + \frac{b^4}{2 \lambda_g} \\
        a^2 (ab) &\le \frac{\lambda_g a^4}{2} + \frac{b^4}{ 2\lambda_g^3}.
    \end{align*}
    Rearranging this, we get
    \begin{align*}
        \EE_{|t}\|y_{t+1} - z_{t+1}\|^4 &\le (1- 3\lambda_g \beta/2) \|y_t - z_t\|^4 \\
        &\quad + 
        \beta^3 \cdot \mathcal{O} \left( \frac{m^2 \sigma_{1,F}^4}{\lambda^4 \lambda_g} + \frac{L_{1,G}^4 m^2}{ \lambda_g} \|y_t-z_t\|^4 + \frac{\mu^4 L_{1,G}^4 m^6}{\lambda_g} \right) \\&\quad + \mathcal{O}\left(\frac{\beta L_{0,f}^4}{\lambda^4 \lambda_g^3}\right) \\
        &\le (1- \lambda_g \beta) \|y_t - z_t\|^4 + \mathcal{O}\left(\frac{\beta L_{0,f}^4}{\lambda^4 \lambda_g^3} \right) + \mathcal{O}(\beta^2),
    \end{align*}
    given that $\beta^2 m^2 L_{1,G}^4 \ll \lambda_g^2$.
\end{proof}

\end{document}